\documentclass[]{article}
\usepackage{amsmath}
\usepackage{amssymb}
\usepackage{graphics}
\usepackage{graphicx}

\usepackage{caption}
\usepackage{subcaption}
\usepackage{algorithmic}
\usepackage{algorithm}
\usepackage{amsfonts}
\usepackage{amsthm}
\usepackage{fullpage}
\usepackage{url}
\usepackage{color}
\usepackage{float}
\usepackage{placeins}

\def\idots{{\reflectbox{$\ddots$}}}
\def\N{{\mathbb{N}}}
\def\R{{\mathbb{R}}}

\def\dy{{\textrm{d}y}}

\newtheorem{lemma}{Lemma}
\newtheorem{theorem}{Theorem}

\newtheorem{definition}{Definition}

\begin{document}

\title{Study of boundary conditions in the Iterative Filtering method for the decomposition of nonstationary signals}

\author{Antonio Cicone\thanks{Istituto Nazionale di Alta Matematica, Citt\`a Universitaria, P.le Aldo Moro 5, 00185, Roma, Italy, and DISIM, Universit\`a degli Studi dell'Aquila, via Vetoio 1, 67100, L'Aquila, Italy, and Gran Sasso Science Institute, viale Francesco Crispi 7, 67100, L'Aquila, Italy ({\tt antonio.cicone@univaq.it})}
\and Pietro Dell'Acqua\thanks{Libera Universit\`a di Bolzano, Piazza Domenicani 3, 39100 Bolzano, Italy ({\tt pietro.dellacqua@gmail.com})}}

\maketitle
\date{}

\begin{abstract}
Non-stationary and non-linear signals are ubiquitous in real life.
Their decomposition and analysis is an important research topic of signal processing.
Recently a new technique, called Iterative Filtering, has been developed with the goal of decomposing such signals into simple oscillatory components.
Several papers have been devoted to the investigation of this technique from a mathematical point of view.
All these works start with the assumption that each compactly supported signal is extended periodically outside the boundaries.
In this work, we tackle the problem of studying the influence of different boundary conditions on the decompositions produced by the Iterative Filtering method.
In particular, the choice of boundary conditions gives rise to different types of structured matrices.
Thus, we describe their spectral properties and then convergence properties of Iterative Filtering algorithm (in which such matrices are involved).
Numerical results provide an interesting overview on important aspects
(such as accuracy and error propagation) of the techniques proposed
and show the way of further promising developments.
\end{abstract}

\section{Introduction}

Given a real life non-stationary signal $\mathbf{s}(x)$, $x\in\R$,
we may be interested in decomposing it into simple components in order to identify features and quasi periodicities hidden in it. We can think, for instance, to a economic index, like the GDP of a nation, or a geophysical signal, like the sea level during a tsunami, as well as to an engineering measure, like the vibrations of a structure or machinery. Standard techniques, like Fourier or wavelet transform, cannot help, in general, to decompose meaningfully non-stationary signals. Whereas, following the idea proposed by Huang et al. in \cite{huang1998empirical}, we can iteratively decompose such signals into a finite sequence of simple components, defined Intrinsic Mode Functions (IMFs), which fulfill two properties:
\textit{i}) the number of extrema and the number of zero crossings must either equal or differ at most by one;
\textit{ii}) considering upper and lower envelopes connecting respectively all the local maxima and minima of the function, their mean has to be zero at any point.

The method originally proposed by Huang et al. in \cite{huang1998empirical}, called Empirical Mode Decomposition (EMD), proved to be unstable to small perturbations. For this reason an alternative algorithm, called Ensemble Empirical Mode Decomposition (EEMD), was proposed in \cite{wu2009ensemble} which is based on the idea of applying EMD to an ensemble of signals produced perturbing hundreds of times the given one with noise. The decomposition is then derived as the average decomposition of the ensemble.
In \cite{lin2009iterative} the authors proposed an alternative technique to the EMD, called Iterative Filtering (IF), which has the very same structure of EMD, but it is stable and convergent both in the continuous setting \cite{cicone2014adaptive,huang2009convergence} and in the discrete one \cite{cicone2017numerical,cicone2017multidimensional}. We point out that IF has been also generalized producing the so called Adaptive Local Iterative Filtering (ALIF) algorithm whose convergence and stability are under investigation \cite{cicone2014adaptive,cicone2017spectral}. For further details on why standard techniques may fail in decomposing a non-stationary signal and the advantages of using EEMD or IF we refer the interested reader to \cite{cicone2017dummies}.

In this work we consider the case of compactly supported and discrete signals. For simplicity we assume that each signal
$\mathbf{s}=\left[\mathbf{s}(x_j)\right]_{j=0}^{n-1}$ (with $n\in\N$)
is supported on $[0,1]$ and it is sampled at $n$ points $x_j= \frac{j}{n-1}$, with $j=0,\ldots,n-1$. Without loosing generality we can further assume that $\| \mathbf{s} \|_2 = 1$.

Any signal decomposition method which deals with a compactly supported signal requires assumptions on how the signal extends outside the boundaries, the so called Boundary Conditions (BCs). This aspect has a key role in many applications, for instance in image restoration \cite{pietro2016}, and usually the goal is to employ BCs able to guarantee good accuracy of the approximated solution computed by some numerical algorithm \cite{pietro2017}.
Regarding IF, in \cite{cicone2017numerical} the authors addressed its convergence when it is assumed that the signals extend periodically outside the boundaries. The questions which are still open are: does IF converge also for other kinds of BCs? Given a signal extended artificially outside the boundaries in a certain way, how do the errors introduced outside the boundaries effect the decomposition in the iterations? Given a compactly supported signal what is the best choice in terms of BCs?

The paper is organized as follows.
In Section \ref{sec:IF} we review the IF algorithm applied to discrete and compactly supported signals.
Section \ref{sec:BC} is devoted to a summary of BCs and their properties.
In Section \ref{sec:IFconv} we study the IF convergence when different BCs are chosen for the signal.
In Section \ref{sec:Error} we present an extended version of IF method,
which is useful for addressing the question of how the errors propagate from outside the boundaries to the inside.
This work ends with some numerical examples showing the impact of BCs on decomposition quality and error propagation in Section \ref{sec:Examples},
and concluding remarks in Section \ref{sec:Conclusions}.

\section{Discrete Iterative Filtering}
\label{sec:IF}

In this section we review the IF method focusing on the case of discrete and compactly supported signals $\mathbf{s}$.
We start from the definition of filter

\begin{definition}\label{def:filter}
A non-negative vector $\mathbf{w}=(0,\ldots,0,w_{-l},\ldots,w_{-1},w_{0},w_1,\ldots,w_{l},0,\ldots,0) \in \R^n$
such that $w_j>0$ for $j=-l,\ldots,l$ and $\underset{j=-l}{\overset{l}{\sum}} w_j = 1$
is called a \textbf{filter} of length $l$, with $0 < l \leq \lfloor \frac{n-1}{2} \rfloor$.
If a filter is such that $w_{-j}=w_j$ for $j=1,\ldots,l$,
then $\mathbf{w}$ is called \textbf{symmetric}.
If a filter is such that for $0 \leq i < j $, $w_{i} \geq w_{j}$
and for $j < i \leq 0 $, $w_{i} \geq w_{j}$,
then $\mathbf{w}$ is called \textbf{decreasing}.
\end{definition}

In this work we consider only symmetric filters.
When we have a symmetric filter associated with some step $m$, we employ the following notation
\begin{equation}
\mathbf{w}_m=(0,\ldots,0,w_{l_m}^m,\ldots,w_{1}^m,w_{0}^m,w_1^m,\ldots,w_{l_m}^m,0,\ldots,0),
\end{equation}
with $w_{0}^m + 2 \underset{j=1}{\overset{l_m}{\sum}} w_j^m = 1$.

If we assume that some filter shape $h:[-1,1] \rightarrow \R$ (symmetric with respect to $y$-axis) has been selected a priori,
like one of the Fokker-Planck filters described in \cite{cicone2014adaptive},
then the elements $w_j^m$ can be computed, for $j=0,1,\ldots,l_m$, by the linear scaling formula
\begin{equation}
w_j^m = h\left(\dfrac{j}{l_m}\right) \dfrac{1}{l_m},
\end{equation}
where $l_m$ is the length that characterizes the filter.
Assuming $\mathbf{s}_1^m=\mathbf{s}$, where the two indices will become clear in the next paragraph, the main step of the IF method,
for $i=0,\ldots,n-1$, is
\begin{eqnarray*}
\mathbf{s}_{k+1}^m(x_i) &=& \mathbf{s}_{k}^m(x_i)-\int_{x_i-\frac{l_m}{n-1}}^{x_i+\frac{l_m}{n-1}} \!\!\!\!\!\!\! \mathbf{s}_k^m(y)
h\left(\dfrac{(x_i-y)(n-1)}{l_m}\right)\dfrac{n-1}{l_m} \dy \\
&\approx& \mathbf{s}_{k}^m(x_i)-\!\!\!\!\! \sum_{x_j=x_i-\frac{l_m}{n-1}}^{x_i+\frac{l_m}{n-1}}\!\!\!\!\! \mathbf{s}_k^m(x_j)
h\left(\dfrac{(x_i-x_j)(n-1)}{l_m}\right)\dfrac{1}{l_m} \\
&=& \mathbf{s}_{k}^m(x_i)-\!\!\! \sum_{j=i-l_m}^{i+l_m} \!\!\! \mathbf{s}_k^m(x_j)
h\left(\dfrac{i-j}{l_m}\right)\dfrac{1}{l_m} \\
&=& \mathbf{s}_{k}^m(x_i)-\!\!\! \sum_{j=i-l_m}^{i+l_m} \!\!\! \mathbf{s}_k^m(x_j) w_{\vert i-j \vert}^m
\end{eqnarray*}

Algorithm \ref{algo:IF_discrete} provides the pseudocode of the Discrete Iterative Filtering (DIF) Algorithm.
We observe that the first while loop is called Outer Loop, whereas the second one Inner Loop.
In the notation $\mathbf{s}_k^m$, $m$ denotes the step relative to the Outer Loop, while $k$ denotes the step relative to the Inner Loop.
\begin{algorithm}
\caption{\textbf{Discrete Iterative Filtering} IMFs = DIF$(s, h)$}\label{algo:IF_discrete}
\begin{algorithmic}
\STATE $m=1$
\STATE $\mathbf{s}_1^m = \mathbf{s}$
\WHILE{the number of extrema of $\mathbf{s}_1^m$ $\geq 2$}
	  \STATE $k=1$
      \STATE compute the filter length $l_m$ for the signal $\mathbf{s}_k^m$
      \STATE compute $\mathbf{w}_m$ (having $h$ and $l_m$)
      \WHILE{the stopping criterion is not satisfied}
                  \STATE $(\mathbf{s}_{k}^m)^{\mathcal{BC}}(x_i)=\mathbf{s}_{k}^m(x_i)$, $i=0,\ldots,n-1$
                  \STATE apply BCs for computing $(\mathbf{s}_{k}^m)^{\mathcal{BC}}(x_i)$, $i=-l_m,\ldots,-1$ and $i=n,\ldots,n-1+l_m$
                  		\STATE  $\mathbf{s}_{k+1}^m(x_i) = \mathbf{s}_{k}^m(x_i) - \underset{j=i-l_m}{\overset{i+l_m}{\sum}} (\mathbf{s}_{k}^m)^{\mathcal{BC}}(x_j) w_{\vert i-j \vert}^m$, $\quad i= 0,\ldots, n-1$
                  \STATE  $k = k+1$
      \ENDWHILE
      \STATE $\mathbf{f}_m = \mathbf{s}_{k}^m$
      \STATE $\mathbf{s}_1^{m+1} = \mathbf{s}_1^{m}-\mathbf{f}_{m}$
      \STATE $m = m+1$
\ENDWHILE
\STATE  $\mathbf{f}_m = \mathbf{s}_1^m$
\STATE IMFs = $\{ \mathbf{f}_1,\ldots, \mathbf{f}_m \}$
\end{algorithmic}
\end{algorithm}
The idea is to suitably choose the filter length $l_m$ in order to capture the desired frequencies in the IMF $\mathbf{f}_m$,
and this is done by subtracting iteratively from the signal its moving average computed as convolution of the signal itself with the selected filter. In matrix form we have
\begin{equation}\label{eq:MatrixForm}
    \mathbf{s}_{k+1}^m=(I-W_m)\mathbf{s}_k^m,
\end{equation}
where $W_m$ is a structured matrix constructed from the filter $\mathbf{w}_m$
and the BCs imposed.
In particular $W_m$ can be written as the sum of two matrices
\begin{equation}
W_{m}^{\mathcal{BC}}=T_{m}+K_{m}^{\mathcal{BC}},
\end{equation}
where the first one is a Toeplitz matrix,
while the second one is a correction matrix which depends on BCs (for details, see Section \ref{sec:BC}).

From (\ref{eq:MatrixForm}), it follows immediately that
\begin{equation}
    \mathbf{s}_{k+1}^m = (I-W_m)^k \mathbf{s}_1^m,
\end{equation}
so ideally the first IMF is given by
\begin{equation}\label{eq:First_IMF_fixed_length}
\mathbf{f}_1 = \lim_{k\rightarrow\infty} (I-W_1)^{k} \mathbf{s}.
\end{equation}
However, in the implemented algorithm we do not let $k$ to go to infinity, instead we use a stopping criterion. We can define, for instance, the following quantity
\begin{equation}\label{eq:SD}
\Delta_k^m:=\frac{\|\mathbf{s}_{k+1}^m-\mathbf{s}_{k}^m\|_2}{\|\mathbf{s}_{k}^m\|_2},
\end{equation}
so we can either stop the process when the value $\Delta_k^m$ reaches a certain threshold or we can introduce a limit on the maximal number of iterations for all the Inner Loops. It is also possible to adopt different stopping criteria for different Inner Loops.

\section{Boundary conditions and structured matrices}
\label{sec:BC}

Boundary conditions deal with the problem of extending the signal outside the field of view in which the detection is made.
So let $\mathbf{s}$ be the signal inside the boundaries and let $p$ be the parameter relative to the space outside the boundaries,
we have that,
for $j= 1, \ldots, p$,
Zero BCs are defined as
\begin{equation}
\mathbf{s}(x_{-j}) = 0, \hspace{0.5cm} \mathbf{s}(x_{n-1+j}) = 0,
\end{equation}
Periodic BCs are defined as
\begin{equation}
\mathbf{s}(x_{-j}) = \mathbf{s}(x_{n-j}), \hspace{0.5cm} \mathbf{s}(x_{n-1+j}) = \mathbf{s}(x_{j-1}),
\end{equation}
Reflective BCs are defined as
\begin{equation}
\mathbf{s}(x_{-j}) = \mathbf{s}(x_{j-1}), \hspace{0.5cm} \mathbf{s}(x_{n-1+j}) = \mathbf{s}_{n-j},
\end{equation}
Anti-Reflective BCs are defined as
\begin{equation}
\label{AR1d}
\mathbf{s}(x_{-j})=2\mathbf{s}(x_{0})-\mathbf{s}(x_{j}), \hspace{0.5cm} \mathbf{s}(x_{n-1+j})=2\mathbf{s}(x_{n-1})-\mathbf{s}(x_{n-1-j}).
\end{equation}

According to the BCs imposed, we have a different kind of structured matrix as $W^{\mathcal{BC}}$,
whose elements are defined from values of the filter $\mathbf{w}$ of length $l$.
As said, here we take into account symmetric filters,
since this choice allows to have useful theoretical properties.
In particular,
symmetry is not necessary to get the algebra of Circulant matrices,
associated with Periodic BCs,
while it is necessary to get the algebra of Reflective matrices and
the algebra of Anti-Reflective matrices.
The symmetry property also allows in all these three cases to have a fast transform
that can be employed for computing matrix-vectors products in an efficient way.
In particular, we have Discrete Fourier Transform (DFT) for Circulant matrices,
Discrete Cosine Transform of type III (DCT-III) for Reflective matrices
and Anti-Reflective Transform (ART)
-- strongly linked to Discrete Sine Transform of type I (DST-I) --
for Anti-Reflective matrices.

Unfortunately, this does not hold for Toeplitz matrices, associated with Zero BCs,
i.e. $W^{\mathcal{Z}}=T$, where
\begin{equation}\label{eq:Toeplitz}
T=\left(
\begin{array}{cccccccccc}
w_{0} & w_{1} & w_{2} & \ldots  & w_{l} &  &  &  &  &  \\
w_{1} & w_{0} & w_{1} & w_{2} & \ddots  & w_{l} &  &  &  &  \\
w_{2} & w_{1} & w_{0} & w_{1} & w_{2} & \ddots  & w_{l} &  &  &  \\
\vdots  & w_{2} & w_{1} & w_{0} & w_{1} & w_{2} & \ddots  & \ddots  &  &  \\
w_{l} & \ddots  & w_{2} & w_{1} & \ddots  & \ddots  & \ddots  & \ddots  & w_{l} &  \\
& w_{l} & \ddots  & w_{2} & \ddots  & \ddots  & w_{1} & w_{2} & \ddots  & w_{l} \\
&  & w_{l} & \ddots  & \ddots  & w_{1} & w_{0} & w_{1} & w_{2} & \vdots  \\
&  &  & \ddots  & \ddots  & w_{2} & w_{1} & w_{0} & w_{1} & w_{2} \\
&  &  &  & w_{l} & \ddots  & w_{2} & w_{1} & w_{0} & w_{1} \\
&  &  &  &  & w_{l} & \ldots  & w_{2} & w_{1} & w_{0}
\end{array}%
\right) _{n\times n}
.
\end{equation}

In case of Periodic BCs, we have Circulant matrices that have this form
\begin{equation}
W^{\mathcal{P}}=\left(
\begin{array}{cccccccccc}
w_{0} & w_{1} & \ldots  & w_{l} &  &  &  & w_{l} & \ldots  & w_{1} \\
w_{1} & w_{0} & w_{1} & \ddots  & w_{l} &  &  &  & \ddots  & \vdots  \\
\vdots  & w_{1} & w_{0} & w_{1} & \ddots  & w_{l} &  &  &  & w_{l}\\
w_{l} & \ddots  & w_{1} & w_{0} & w_{1} & \ddots  & \ddots  &  &  &  \\
& w_{l} & \ddots  & w_{1} & w_{0} & \ddots  & \ddots  & w_{l}&  &  \\
&  & w_{l} & \ddots  & \ddots  & \ddots  & w_{1} & \ddots  & w_{l} &  \\
&  &  & \ddots  & \ddots  & w_{1} & w_{0} & w_{1} & \ddots  & w_{l} \\
w_{l} &  &  &  & w_{l} & \ddots  & w_{1} & w_{0} & w_{1} & \vdots  \\
\vdots  & \ddots  &  &  &  & w_{l} & \ddots  & w_{1} & w_{0} & w_{1} \\
w_{1} & \ldots  & w_{l} &  &  &  & w_{l} & \ldots  & w_{1} & w_{0}%
\end{array}%
\right) _{n\times n}
.
\end{equation}
Denoted by $\mathrm{i}$ the imaginary unit,
$W^{\mathcal{P}}$ can be diagonalized by $Q^{\mathcal{P}}$, defined as
\begin{equation}
[Q^{\mathcal{P}}]_{i,j} = \sqrt{n} F^{-1} =\dfrac{1}{\sqrt{n}} e^{ \frac{2 (i-1) (j-1) \pi \mathrm{i} }{n} }, \hspace{1cm} i,j=1,\ldots,n,
\end{equation}
where $F$ is the $n$-dimensional DFT.
Eigenvalues of $W^{\mathcal{P}}$ can be computed by this formula, for $i=1,\ldots,n$,
\begin{equation}
\label{equ:eigP}
\lambda _{i}^{\mathcal{P}}=w_{0}+2\underset{j=1}{\overset{l}{\sum }}w_{j}\cos\left( \frac{2j(i-1)\pi }{n}\right).
\end{equation}

In the Reflective case, we have $W^{\mathcal{R}}=T+H^{\mathcal{R}}$,
where $H^{\mathcal{R}}$ is a Hankel matrix
\begin{equation}
H^{\mathcal{R}}=\left(
\begin{array}{cccccccccc}
w_{1} & w_{2} & w_{3} & \ldots  & w_{l} &  &  &  &  &  \\
w_{2} & w_{3} & \idots  & \idots  &  &  &  &  &  &  \\
w_{3} & \idots  & \idots  &  &  &  &  &  &  &  \\
\vdots  & \idots  &  &  &  &  &  &  &  &  \\
w_{l} &  &  &  &  &  &  &  &  &  \\
&  &  &  &  &  &  &  &  & w_{l} \\
&  &  &  &  &  &  &  & \idots  & \vdots  \\
&  &  &  &  &  &  & \idots  & \idots  & w_{3} \\
&  &  &  &  &  & \idots  & \idots  & w_{3} & w_{2} \\
&  &  &  &  & w_{l} & \ldots  & w_{3} & w_{2} & w_{1}%
\end{array}%
\right) _{n\times n}
.
\end{equation}
$W^{\mathcal{R}}$ can be diagonalized by $Q^{\mathcal{R}}$,
that is the $n$-dimensional DCT-III,
having entries
\begin{equation}
[Q^{\mathcal{R}}]_{i,j}=\sqrt{\dfrac{2-\delta _{i1}}{n}}\cos \left( \dfrac{%
(i-1)(2j-1)\pi }{2n}\right), \hspace{1cm} i,j=1,\ldots,n,
\end{equation}
where $\delta_{ij}$ denotes the Kronecker delta.
Eigenvalues of $W^{\mathcal{R}}$ can be obtained by the following formula
\begin{equation}
\lambda _{i}^{\mathcal{R}}=\dfrac{[Q^{\mathcal{R}}(W^{\mathcal{R}} \mathbf{e}_{1})]_{i}}{[Q\mathbf{e}_{1}]_{i}},
\end{equation}
where $\mathbf{e}_{1}=(1,0,\ldots ,0)^{T} \in \R^n$.
Therefore,
using the variable $t=\frac{(i-1)\pi }{2n}$,
and exploiting
\begin{equation*}
\cos (2jt-t)=\cos (2jt)\cos (t)+\sin (2jt)\sin(t)
\end{equation*}
and
\begin{equation*}
\cos (2jt+t)=\cos (2jt)\cos (t)-\sin (2jt)\sin(t),
\end{equation*}
we get that in the Reflective case eigenvalues can be computed by this formula, for $i=1,\ldots,n$,
\begin{eqnarray}
\label{equ:eigR}
\lambda _{i}^{\mathcal{R}} &=&\underset{j=1}{\overset{n}{\sum }}[W^{\mathcal{R}}]_{1,j}\dfrac{\cos \left( (2j-1)\dfrac{(i-1)\pi }{2n}\right) }{\cos \left( \dfrac{(i-1)\pi }{2n}\right) } \nonumber \\
&=&w_{0}+\underset{j=1}{\overset{l}{\sum }}w_{j}\dfrac{\cos((2j-1)t)+\cos ((2j+1)t)}{\cos \left( t\right) } \nonumber \\
&=&w_{0}+2\underset{j=1}{\overset{l}{\sum }}w_{j}\cos (2jt) \nonumber \\
&=&w_{0}+2\underset{j=1}{\overset{l}{\sum }}w_{j}\cos \left( \frac{j(i-1)\pi }{n}\right)
\end{eqnarray}

In the Anti-Reflective case, the structure of the matrix is more involved, namely
\begin{equation}
W^{\mathcal{AR}}=\left(
\begin{array}{ccccccc}
z_{1}+w_{0} & 0 & \ldots & \ldots & \ldots & 0 & 0 \\
z_{2}+w_{1} &  &  &  &  &  & \vdots \\
\vdots &  &  &  &  &  & \vdots \\
z_{l}+w_{l-1} &  &  &  &  &  & 0 \\
w_{l} &  &  & \hat{W}^{\mathcal{AR}} &  &  & w_{l} \\
0 &  &  &  &  &  & z_{l}+w_{l-1}\\
\vdots &  &  &  &  &  & \vdots \\
\vdots &  &  &  &  &  & z_{2}+w_{1} \\
0 & 0 & \ldots & \ldots & \ldots & 0 & z_{1}+w_{0}
\end{array}%
\right) _{n\times n}
,
\end{equation}
where $z_{j}= 2 \underset{k=j}{\overset{l}{\sum }}w_{k}$ and $\hat{W}^{\mathcal{AR}}= P W^{\mathcal{AR}}P^{T}$, with
\begin{equation}
P=\left(
\begin{array}{ccccccc}
0 & 1 &  &  &  &  & 0 \\
0 &  & 1 &  &  &  & 0 \\
\vdots &  &  & \ddots &  &  & \vdots \\
0 &  &  &  & 1 &  & 0 \\
0 &  &  &  &  & 1 & 0%
\end{array}%
\right) _{(n-2)\times n}
.
\end{equation}

Moreover, $\hat{W}^{\mathcal{AR}}=\hat{T}-\hat{H}^{\mathcal{AR}}$, where $\hat{T}$ is a Toeplitz matrix
\begin{equation}
\hat{T}=\left(
\begin{array}{cccccccc}
w_{0} & w_{1} & \ldots  & w_{l} &  &  &  &  \\
w_{1} & w_{0} & w_{1} & \ddots  & w_{l} &  &  &  \\
\vdots  & w_{1} & w_{0} & w_{1} & \ddots  & \ddots  &  &  \\
w_{l} & \ddots  & w_{1} & w_{0} & \ddots  & \ddots  & w_{l} &  \\
& w_{l} & \ddots  & \ddots  & \ddots  & w_{1} & \ddots  & w_{l} \\
&  & \ddots  & \ddots  & w_{1} & w_{0} & w_{1} & \vdots  \\
&  &  & w_{l} & \ddots  & w_{1} & w_{0} & w_{1} \\
&  &  &  & w_{l} & \ldots  & w_{1} & w_{0}%
\end{array}%
\right) _{(n-2)\times (n-2)}
,
\end{equation}
while $\hat{H}^{\mathcal{AR}}$ is a Hankel matrix
\begin{equation}
\hat{H}^{\mathcal{AR}}=\left(
\begin{array}{cccccccc}
w_{2} & w_{3} & \ldots  & w_{l} &  &  &  &  \\
w_{3} & \idots  & \idots  &  &  &  &  &  \\
\vdots  & \idots  &  &  &  &  &  &  \\
w_{l} &  &  &  &  &  &  &  \\
&  &  &  &  &  &  & w_{l} \\
&  &  &  &  &  & \idots  & \vdots  \\
&  &  &  &  & \idots  & \idots  & w_{3} \\
&  &  &  & w_{l} & \ldots  & w_{3} & w_{2}
\end{array}%
\right) _{(n-2)\times (n-2)}
.
\end{equation}
$\hat{W}^{\mathcal{AR}}$ can be diagonalized by $\hat{Q}^{\mathcal{AR}}$,
that is the ($n-2$)-dimensional DST-I,
having entries
\begin{equation}
[\hat{Q}^{\mathcal{AR}}]_{i,j}=\sqrt{\dfrac{2}{n-1}}\sin \left( \dfrac{ij\pi }{n-1}\right), \hspace{1cm} i,j=1,\ldots,n-2.
\end{equation}
Eigenvalues of $\hat{W}^{\mathcal{AR}}$ can be obtained by the following formula
\begin{equation}
\hat{\lambda}_{i}^{\mathcal{AR}}=\dfrac{[\hat{Q}^{\mathcal{AR}}(\hat{W}^{\mathcal{AR}}\mathbf{\hat{e}}_{1})]_{i}}{[\hat{Q}^{\mathcal{AR}}\mathbf{\hat{e}}_{1}]_{i}}
\end{equation}
where $\mathbf{\hat{e}}_{1}=(1,0,\ldots ,0)^{T} \in \R^{n-2}$.
Therefore,
using the variable $t=\frac{i\pi }{n-1}$,
and exploiting
\begin{equation*}
\sin (2t)=2\sin (t)\cos (t)
\end{equation*}
and
\begin{equation*}
\cos (2t)=1-2\sin ^{2}(t)
\end{equation*}
and
\begin{eqnarray*}
\sin ((j+2)t)-\sin jt) &=& \sin (jt)\cos (2t)+\sin (2t)\cos (jt)-\sin(jt) \\
&=& \sin (jt)(1-2\sin ^{2}(t))+2\sin (t)\cos (t)\cos (jt)-\sin (jt) \\
&=& -2\sin(jt)\sin ^{2}(t)+2\sin (t)\cos (t)\cos (jt)
\end{eqnarray*}
and
\begin{equation*}
2 \sin(jt) \sin(t) = \cos(jt-t) - \cos(jt+t)
\end{equation*}
and\begin{equation*}
2 \cos(t) \cos(jt) = \cos(jt-t) - \cos(jt+t)
\end{equation*}
we get that eigenvalues of $\hat{W}^{\mathcal{AR}}$ can be computed by this formula, for $i=1,\ldots,n-2$,
\begin{eqnarray}
\label{equ:eigAR}
\hat{\lambda}_{i}^{\mathcal{AR}} &=&\underset{j=1}{\overset{n-2}{\sum }}[\hat{W}^{\mathcal{AR}}]_{1,j}\dfrac{\sin
\left( \dfrac{ji\pi }{n-1}\right) }{\sin \left( \dfrac{i\pi }{n-1}\right) } \nonumber \\
&=&w_{0}+w_{1}\dfrac{\sin (2t)}{\sin \left( t\right) }+\underset{j=1}{\overset{l-1}{\sum }}w_{j+1}\dfrac{\sin ((j+2)t)}{\sin \left( t\right) }-w_{j+1}\dfrac{\sin (jt)}{\sin \left( t\right) } \nonumber \\
&=&w_{0}+2w_{1}\cos (t)+\underset{j=1}{\overset{l-1}{\sum }}w_{j+1}\left(2\cos (t)\cos (jt)-2\sin (jt)\sin (t)\right)  \nonumber \\
&=&w_{0}+2w_{1}\cos (t)+2\underset{j=1}{\overset{l-1}{\sum }}w_{j+1}\cos((j+1)t) \nonumber \\
&=&w_{0}+2w_{1}\cos (t)+2\underset{j=2}{\overset{l}{\sum }}w_{j}\cos(jt) \nonumber \\
&=&w_{0}+2\underset{j=1}{\overset{l}{\sum }}w_{j}\cos \left(\dfrac{ji\pi }{n-1} \right)
\end{eqnarray}
By Lemma 3.1 in \cite{serra2003antirefl}, we know that the eigenvalues of $W^{\mathcal{AR}}$ are given by 1 with multiplicity
two and by the eigenvalues $\lbrace \hat{\lambda}_{i}^{\mathcal{AR}} \rbrace_{i=1,\ldots,n-2}$ of $\hat{W}^{\mathcal{AR}}$.

Looking at the structure of the matrices presented,
it can be easily verified that
the eigenvector $u_1^{\mathcal{P}}$ associated with $\lambda_1^{\mathcal{P}}=1$ is $(1,\ldots,1)^T$,
the eigenvector $u_1^{\mathcal{R}}$ associated with $\lambda_1^{\mathcal{R}}=1$ is $(1,\ldots,1)^T$,
and the eigenvectors associated with $\lambda_1^{\mathcal{AR}}=\lambda_2^{\mathcal{AR}}=1$ are
$u_1^{\mathcal{AR}}=(0,1,2,\ldots,n-2,n-1)^T$ and $u_2^{\mathcal{AR}}=(n-1,n-2,\ldots,2,1,0)^T$.

We notice that $Q^{\mathcal{P}}$, $Q^{\mathcal{R}}$  and $\hat{Q}^{\mathcal{AR}}$ are all unitary matrices.
However, the last one is linked to $\hat{W}^{\mathcal{AR}}$,
so we also need to know how to diagonalize the original matrix $W^{\mathcal{AR}}$.
In other words, basing on DST-I, we need to introduce ART.
Unfortunately this transform is not unitary, since the matrices in Anti-Reflective algebra are in general not normal.
Recalling theoretical results reported in \cite{arico2011,pietro2016},
we have that $W^{\mathcal{AR}}$ can be diagonalized by $Q^{\mathcal{AR}}$,
that is the $n$-dimensional ART, defined as follows
\begin{equation}
\label{eq:ART}
Q^{\mathcal{AR}}=\left(
\begin{array}{ccccccc}
(n-1) \eta^{-1} & 0 & \ldots & \ldots & \ldots & 0 & 0 \\
(n-2) \eta^{-1}  &  &  &  &  &  & \eta^{-1}  \\
(n-3) \eta^{-1}  &  &  &  &  &  & 2 \eta^{-1}  \\
\vdots  &  &  &  &  &  & \vdots \\
\vdots  &  &  & \hat{Q}^{\mathcal{AR}} &  &  & \vdots \\
\vdots  &  &  &  &  &  & \vdots\\
2 \eta^{-1}  &  &  &  &  &  & (n-3) \eta^{-1}  \\
\eta^{-1}  &  &  &  &  &  & (n-2) \eta^{-1}  \\
0 & 0 & \ldots & \ldots & \ldots & 0 & (n-1) \eta^{-1}
\end{array}%
\right) _{n\times n}
,
\end{equation}
where $\eta = \sqrt{\sum_{j=0}^{n-1} j^2}$ is introduced just to normalize the first and the last column,
which come from eigenvectors  $u_1^{\mathcal{AR}}$ and $u_2^{\mathcal{AR}}$.

Summarizing, in this Section we have introduced BCs,
analyzing the properties of matrices associated with their choice,
in particular providing descriptions of matrix structures,
formulas for computation of eigenvalues
and diagonalization results (associated with discrete transforms).
All this will be very useful in Section \ref{sec:IFconv},
in which convergence properties of DIF will be investigated.

\section{Spectral properties and convergence results}
\label{sec:IFconv}

At this stage, we consider some BCs (Periodic or Reflective or Anti-Reflective).
By exploiting results presented in Section \ref{sec:BC},
we are able to prove spectral properties of $W^{\mathcal{BC}}$ and convergence results for DIF algorithm.

\begin{lemma}
\label{lemma:1}
Let $W^{\mathcal{BC}} \in \R^{n \times n}$ be the structured matrix constructed from a symmetric decreasing filter $\mathbf{w}$ of length
$0 < l \leq \lfloor \frac{n-1}{2} \rfloor$, imposing some BCs (Periodic or Reflective or Anti-Reflective).
Then $\sigma (W^{\mathcal{BC}})\subseteq [-1,1]$.
\end{lemma}
\begin{proof}
By definition $W^{\mathcal{BC}}$ is a symmetric matrix, so it has a real spectrum.
By considering the following estimate relative to the spectral radius
\begin{equation*}
\rho (W^{\mathcal{BC}})\leq \left\Vert W^{\mathcal{BC}}\right\Vert _{\infty}=
\underset{i}{\max }\underset{j=1}{\overset{n}{\sum }}\left\vert[W^{\mathcal{BC}}]_{i,j}\right\vert=
w_{0}+2\underset{k=1}{\overset{l}{\sum }}w_{k}=1,
\end{equation*}
we can conclude that all eigenvalues lie in the interval $[-1,1]$.
\end{proof}

\begin{theorem}
\label{teo:1}
Let $\mathbf{s} \in \R^n$ the signal that has to be decomposed.
Let $\mathbf{v}$ be a symmetric decreasing filter of length $0 < l' \leq \lfloor \frac{n-1}{4} \rfloor$
and $\mathbf{w} = \mathbf{v} \ast \mathbf{v}$ be another symmetric decreasing filter of length $0 < l \leq \lfloor \frac{n-1}{2} \rfloor$
defined by making the convolution of $\mathbf{v}$ with itself.
Then for the matrix $W^{\mathcal{BC}} \in \R^{n \times n}$ constructed from $\mathbf{w}$ and some BCs
(Periodic or Reflective or Anti-Reflective),
it holds that $\sigma (W^{\mathcal{BC}})\subseteq [0,1]$.
Moreover,
$\lambda_1^{\mathcal{P}}=1$ and $\lbrace \lambda_{i}^{\mathcal{P}} \rbrace_{i=2,\ldots,n} \subseteq [0,1)$,
$\lambda_1^{\mathcal{R}}=1$ and $\lbrace \lambda_{i}^{\mathcal{R}} \rbrace_{i=2,\ldots,n} \subseteq [0,1)$,
$\lambda_1^{\mathcal{AR}}=\lambda_2^{\mathcal{AR}}=1$ and $\lbrace \lambda_{i}^{\mathcal{AR}} \rbrace_{i=3,\ldots,n} \subseteq [0,1)$.
\end{theorem}
\begin{proof}
As already said,
in case of Periodic or Reflective or Anti-Reflective BCs we are in a matrix algebra;
this means that the matrix product gives rise to a matrix
which can still be interpreted as a matrix constructed from a filter and the same BCs.
In particular,
thanks to the fact that $l' \leq \lfloor \frac{n-1}{4} \rfloor$,
such filter can be computed by means of convolution.
Thus, if we denote by $V^{\mathcal{BC}}$ the structured matrix associated with filter $\mathbf{v}$,
we have that $W^{\mathcal{BC}} = (V^{\mathcal{BC}})^2$.
By Lemma \ref{lemma:1}, $\sigma (V^{\mathcal{BC}})\subseteq [-1,1]$,
therefore $\sigma (W^{\mathcal{BC}})\subseteq [0,1]$.

Moreover,
in case of Periodic BCs,
if we consider (\ref{equ:eigP}) with $i=1$, we get $\lambda_1^{\mathcal{P}}=1$,
while for $1 < i \leq n$ we have a convex combination of cosine values
(each of them strictly less than 1),
except for the first term multiplied by $w_0$ which is equal to 1.
Clearly this sum cannot equal to 1, so it has a value $0 \leq x < 1$.

In case of Reflective BCs,
if we consider (\ref{equ:eigR}) with $i=1$, we get $\lambda_1^{\mathcal{R}}=1$,
while for $1 < i \leq n$ we have a convex combination of cosine values
(each of them strictly less than 1),
except for the first term multiplied by $w_0$ which is equal to 1.
Clearly this sum cannot equal to 1, so it has a value $0 \leq x < 1$.

In case of Anti-Reflective BCs,
if we consider (\ref{equ:eigAR}) for $i=1,\ldots,n-2$
we have a convex combination of cosine values (each of them strictly less than 1),
except for the first term multiplied by $w_0$ which is equal to 1.
Clearly this sum cannot equal to 1, so it has a value $0 \leq x < 1$.
Therefore $n-2$ eigenvalues of $W^{\mathcal{BC}}$ belong to the interval $[0,1)$,
while the remaining two are equal to 1, as we already know.
\end{proof}

In the following Lemma we summarize diagonalization results presented in Section \ref{sec:BC}.

\begin{lemma}
\label{lemma:2}
Let $W^{\mathcal{BC}} \in \R^{n \times n}$ be a matrix constructed from
a symmetric filter $\mathbf{w}$ and some BCs (Periodic or Reflective or Anti-Reflective).
Then
\begin{equation}
W^{\mathcal{BC}}  = Q^{\mathcal{BC}} D^{\mathcal{BC}} (Q^{\mathcal{BC}} )^{-1},
\end{equation}
i.e. $W^{\mathcal{BC}}$ can be diagonalized by $Q^{\mathcal{BC}}$,
whose columns are eigenvectors of $W^{\mathcal{BC}}$ .
\end{lemma}

Now we are ready to discuss the convergence of DIF algorithm. The method in the limit produces IMFs that are projections of the given signal $\mathbf{s} $ onto the eigenspace of $W^{\mathcal{BC}}$ corresponding to the zero eigenvalue which has algebraic and geometric multiplicity $\zeta\in\{0,\ 1,\ldots,\ n-1\}$. Clearly, if  $W^{\mathcal{BC}}$ has only a trivial kernel then the method converges to the zero vector. On the other hand, if we enforce a stopping criterion, the techniques gets approximated IMFs after finitely many steps.
Theorem \ref{teo:2} generalizes theoretical results provided in \cite{cicone2017numerical} only for the case of Periodic BCs.

\begin{theorem}
\label{teo:2}
Let $\mathbf{s} \in \R^n$ be the signal that has to be decomposed.
Let $\mathbf{v}$ be a symmetric decreasing filter of length $0 < l' \leq \lfloor \frac{n-1}{4} \rfloor$
and $\mathbf{w} = \mathbf{v} \ast \mathbf{v}$ be another symmetric decreasing filter of length $0 < l \leq \lfloor \frac{n-1}{2} \rfloor$
defined by making the convolution of $\mathbf{v}$ with itself.
Let $W^{\mathcal{BC}} \in \R^{n \times n}$ be the matrix constructed from $\mathbf{w}$ and some BCs (Periodic or Reflective or Anti-Reflective),
which can be diagonalized by  $Q^{\mathcal{BC}} $,
and $\zeta$ be the number (in the set $\{0,\ 1,\ldots,\ n-1\}$) of its zero eigenvalues.
Let $\alpha^{\mathcal{BC}}$ and $\beta^{\mathcal{BC}}$ be two constants depending on BCs at hand ($\alpha^{\mathcal{P}}=\alpha^{\mathcal{R}}=1$, $\alpha^{\mathcal{AR}}=3$ and $\beta^{\mathcal{P}}=\beta^{\mathcal{R}}=1$, $\beta^{\mathcal{AR}}=2$).

Then, at step $k$ of the inner loop in the DIF method, the first IMF is given by
\begin{equation*}
\mathbf{f}_1 = Q^{\mathcal{BC}} (Z^{\mathcal{BC}} )^k  (Q^{\mathcal{BC}})^{-1} \mathbf{s},
\end{equation*}
where $Z^{\mathcal{BC}}  = I -  D^{\mathcal{BC}}$ can be rewritten in this way
\begin{equation*}
Z^{\mathcal{BC}}  = P \left(
                                          \begin{array}{ccccccc}
                                             1-\lambda^{\mathcal{BC}}_1  &   &   &   &   &   &   \\
                                              &  1-\lambda^{\mathcal{BC}}_2  &   &   &   &   &   \\
                                              &   & \ddots  &   &   &   &   \\
                                              &   &   &  1-\lambda^{\mathcal{BC}}_{n-\zeta}  &   &   &   \\
                                              &   &   &   &  1 &   &   \\
                                              &   &   &   &   &  \ddots &   \\
                                              &   &   &   &   &   & 1  \\
                                          \end{array}
                                        \right) P^T
\end{equation*}
by means of a suitable permutation matrix $P$, and the first $\beta^{\mathcal{BC}}$ elements that appear on the diagonal are equal to zero.
Letting $k$ go to infinity, the first outer loop step of the DIF method converges to
\begin{equation*}
\mathbf{f}_1 = Q^{\mathcal{BC}}  Z_{\infty}  (Q^{\mathcal{BC}})^{-1} \mathbf{s},
\end{equation*}
where $Z_{\infty}$
is a diagonal matrix with entries all zero,
except $\zeta$ diagonal elements equal to one.

Moreover,
fixed $\delta>0$,
for the minimum $k_0\in\N$ such that it holds true the inequality
\begin{equation*}
    \frac{k_0^{k_0}}{\left(k_0+1\right)^{k_0+1}}<\frac{\delta}{\alpha^{\mathcal{BC}} \|(Q^{\mathcal{BC}})^{-1} \mathbf{s}\|_\infty{\sqrt{n-\beta^{\mathcal{BC}}-\zeta}}},
\end{equation*}
we have that the following stopping criterion is satisfied
\begin{equation*}
\left\| \mathbf{s}_{k+1}^1-\mathbf{s}_k^1\right\|_{2}<\delta, \hspace{0.5cm} \forall k\geq k_0,
\end{equation*}
so the DIF algorithm takes finitely many steps to produce the first IMF.
\end{theorem}

\begin{proof}
The first part of the theorem follows from the definition of DIF method in matrix form and from Lemma \ref{lemma:2}.
In fact,
at step $k$ of the inner loop in the DIF method, the first IMF is given by
\begin{equation*}
\mathbf{f}_1 = (I-W^{\mathcal{BC}})^k \mathbf{s} = Q^{\mathcal{BC}} (Z^{\mathcal{BC}} )^k  (Q^{\mathcal{BC}})^{-1} \mathbf{s},
\end{equation*}
Moreover, by Theorem \ref{teo:1} we know that $W^{\mathcal{BC}}$ has $\beta^{\mathcal{BC}}$ eigenvalues equal to 1.
On the other hand, letting $k$ go to infinity, we get
\begin{equation*}
\mathbf{f}_1=\lim_{k\rightarrow \infty} Q^{\mathcal{BC}} (Z^{\mathcal{BC}} )^k (Q^{\mathcal{BC}})^{-1} \mathbf{s}=
Q^{\mathcal{BC}}  Z_{\infty}  (Q^{\mathcal{BC}})^{-1} \mathbf{s}.
\end{equation*}

The second part of the theorem follows from the next expression
\small
\begin{eqnarray*}
 \|\mathbf{s}_{k+1}^1 - \mathbf{s}_k^1\|_2 &=& \|(I-W^{\mathcal{BC}})^{k+1}\mathbf{s} - (I-W^{\mathcal{BC}})^{k} \mathbf{s}\|_2 \\
 &=&\| Q^{\mathcal{BC}} (I-D^{\mathcal{BC}})^{k+1} (Q^{\mathcal{BC}})^{-1}  \mathbf{s} - Q^{\mathcal{BC}} (I-D^{\mathcal{BC}})^k (Q^{\mathcal{BC}})^{-1}  \mathbf{s} \|_2 \\
 &=& \| Q^{\mathcal{BC}} (Z^{\mathcal{BC}})^k ((I-D^{\mathcal{BC}})-I) (Q^{\mathcal{BC}})^{-1}  \mathbf{s}  \|_2 \\
 &=& \| Q^{\mathcal{BC}} (Z^{\mathcal{BC}})^k D^{\mathcal{BC}} (Q^{\mathcal{BC}})^{-1}  \mathbf{s}  \|_2 \\
 &\leq& \alpha^{\mathcal{BC}}  \| (Z^{\mathcal{BC}})^k D^{\mathcal{BC}} (Q^{\mathcal{BC}})^{-1}  \mathbf{s}  \|_2 \\
 &\leq&  \alpha^{\mathcal{BC}}  \left\|P \left(
                                          \begin{array}{cccccccc}
                                             (1-\lambda^{\mathcal{BC}}_1 )^k \lambda^{\mathcal{BC}}_1   &   &   &   &   &  &  \\
                                              &   \ddots   &   &   &   & \\
                                              &   & (1-\lambda^{\mathcal{BC}}_{n-\zeta} )^k \lambda^{\mathcal{BC}}_{n-\zeta}  &   &  & \\
                                              &    &   & 0  &  & \\
                                              &    &   &   & \ddots & \\
                                              &    &   &   &  & 0 \\
                                          \end{array}
                                        \right) P^T\left(
                                                 \begin{array}{c}
                                                   \| (Q^{\mathcal{BC}})^{-1}  \mathbf{s}  \|_\infty \\
                                                   \vdots \\
                                                   \| (Q^{\mathcal{BC}})^{-1}  \mathbf{s} \|_\infty \\
                                                 \end{array}
                                               \right)
                                        \right\|_2 \\
  &\leq & \alpha^{\mathcal{BC}}  {\sqrt{n-\beta^{\mathcal{BC}} -\zeta}} \left(1-\frac{1}{k+1} \right)^k \frac{1}{k+1}  \| (Q^{\mathcal{BC}})^{-1}  \mathbf{s} \|_\infty
\end{eqnarray*}
\normalsize
where $P$ is a suitable permutation matrix and $\alpha^{\mathcal{BC}}$ depends on the value of $\| Q^{\mathcal{BC}}  \|_2$.
So it is equal to $1$ for Periodic and Reflective BCs, since $Q^{\mathcal{P}}$ and $Q^{\mathcal{R}}$ are unitary matrices,
while in the Anti-Reflective case it is equal to $3$ because we can rewrite $Q^{\mathcal{AR}}$
-- see (\ref{eq:ART}) --
as the sum of three matrices,
having non-zero elements only in the first column, in the last column and in the central part
(corresponding to $\hat{Q}^{\mathcal{AR}}$, which is unitary),
then we can apply triangle inequality, exploiting also the fact that the Euclidean norm of each splitting matrix is equal to $1$.
For writing the last inequality, we use the fact that the function $(1-\lambda)^k \lambda$ achieves its maximum at $\lambda=\frac{1}{k+1}$ for $\lambda\in[0,\ 1]$.
Hence the stopping criterion is fulfilled for $k_0$ minimum natural number such that
$\frac{k_0^{k_0}}{\left(k_0+1\right)^{k_0+1}}<\frac{\delta}{\alpha^{\mathcal{BC}} \|(Q^{\mathcal{BC}})^{-1} \mathbf{s}\|_\infty{\sqrt{n-\beta^{\mathcal{BC}}-\zeta}}}$.
\end{proof}

We make here the observation that with the current version of the algorithm we have to reimpose the BCs at every iteration of the inner loop.
In order to reduce the error and its propagation, the idea is to impose the BCs only in the first step of each inner loop, and then let the solution to evolve freely.
This is the reason for introducing the extended approach.

\section{Extended Iterative Filtering}
\label{sec:Error}

Now, once defined the matrix
\begin{equation}
R = \left(\ O_{n\times p}\; I_{n\times n}\; O_{n\times p}\ \right)_{n\times (n+2p)},
\end{equation}
where $O$ is matrix of all zeros and $I$ is the identity matrix, in Algorithm \ref{algo:IF_discrete} we present the pseudocode of Extended Iterative Filtering (EIF) algorithm.
Differently from the DIF method,
it is not based (as DIF algorithm) on a structured matrix $W_{m}^{\mathcal{BC}}\in \R^{n\times n}$
(whose structure depends on BCs),
associated with the original signal $\mathbf{s} \in \R$,
but on a Circulant matrix $\ddot{W}_m^{\mathcal{P}} \in \R^{(n+2p) \times (n+2p)}$
(whose structure does not depend on BCs),
associated with a signal $\mathbf{s}^{\mathcal{BC}} \in \R^{(n+2p)}$
extended by means of any BCs.
Hence, since in this framework the BCs do not affect the matrix structure, we have more freedom in their choice.
Any method that extends the original signal in a suitable way outside the boundaries can be employed now.
Moreover, we notice that in Algorithm \ref{algo:EIF} there is only one step involving the use of BCs,
while in Algorithm \ref{algo:IF_discrete} the BCs are employed at every step of the Inner Loop.
Considering that every time we make use of BCs we are adding errors to the decomposition method,
we think that the proposed approach may improve the standard one.
Furthermore, we remark that employing Periodic BCs allows to reduce the computational burden of IF.
In fact the algorithm can be rewritten using Fast Fourier Transform (FFT)
to become what is known as Fast Iterative Filtering (FIF) \cite{cicone2017numerical}.

\begin{algorithm}
\caption{\textbf{Extended Iterative Filtering} IMFs = EIF$(s, h)$}\label{algo:EIF}
\begin{algorithmic}
\STATE $\mathbf{s}^{\mathcal{BC}}(x_i) = \mathbf{s}(x_i)$, $i=0,\ldots,n-1$
\STATE apply BCs to compute $\mathbf{s}^{\mathcal{BC}}(x_i)$, $i=-p,\ldots,-1$ and $i=n,\ldots,n-1+p$
\STATE $m=1$
\STATE $\mathbf{s}_1^m = \mathbf{s}^{\mathcal{BC}}$
\WHILE{the number of extrema of $\mathbf{s}_1^m$ $\geq 2$}
	  \STATE $k=1$	
      \STATE compute the filter length $l_m$ for the signal $\mathbf{s}_k^m$
      \STATE compute $\mathbf{w}_m$ (having $h$ and $l_m$)
      \WHILE{the stopping criterion is not satisfied}
                  \STATE  $\mathbf{s}_{k+1}^m = (I-\ddot{W}_m^{\mathcal{P}}) \mathbf{s}_{k}^m$
                  \STATE  $k = k+1$
      \ENDWHILE
      \STATE $\mathbf{f}_m = \mathbf{s}_{k}^m$
      \STATE $\mathbf{s}_1^{m+1} = \mathbf{s}_1^{m}-\mathbf{f}_{m}$
      \STATE $m = m+1$
\ENDWHILE
\STATE  $\mathbf{f}_m = \mathbf{s}_1^m$
\STATE IMFs = $\{ R \mathbf{f}_1,\ldots, R \mathbf{f}_m \}$
\end{algorithmic}
\end{algorithm}

We denote by $\mathbf{f}_1\in \R^n$ the first approximated IMF computed by the EIF algorithm.
$\mathbf{\bar{f}}_1\in\R^n$ represents the first exact IMF, whereas $\mathbf{s}^{\mathcal{BC}}\in\R^{n+2p}$ the signal extended according to some a priori chosen boundary conditions and $\ddot{\mathbf{s}}\in\R^{n+2p}$ the unknown real signal extended also outside the boundaries.
Then we can split these last two vectors considering the portion of signal inside and outside the field of view identified by the boundaries, and get
\begin{equation}
\mathbf{s}^{\mathcal{BC}} = \mathbf{s}_{\textrm{in}}^{\mathcal{BC}} + \mathbf{s}_{\textrm{out}}^{\mathcal{BC}}
\end{equation}
and
\begin{equation}
\ddot{\mathbf{s}} = \ddot{\mathbf{s}}_{\textrm{in}} + \ddot{\mathbf{s}}_{\textrm{out}}.
\end{equation}
Clearly $\mathbf{s}_{\textrm{in}}^{\mathcal{BC}} = \ddot{\mathbf{s}}_{\textrm{in}}$.

We have that
\begin{eqnarray*}
\mathbf{f}_1  &=& R (I-\ddot{W}_1^{\mathcal{P}})^{k} \mathbf{s}^{\mathcal{BC}} \\
&=& R (I-\ddot{W}_1^{\mathcal{P}})^{k} (\mathbf{s}_{\textrm{in}}^{\mathcal{BC}} + \mathbf{s}_{\textrm{out}}^{\mathcal{BC}} + \ddot{\mathbf{s}}_{\textrm{out}}-\ddot{\mathbf{s}}_{\textrm{out}}) \\
&=& R (I-\ddot{W}_1^{\mathcal{P}})^{k} \ddot{\mathbf{s}} + R (I-\ddot{W}_1^{\mathcal{P}})^{k} (\mathbf{s}_{\textrm{out}}^{\mathcal{BC}} - \ddot{\mathbf{s}}_{\textrm{out}})
\end{eqnarray*}
where the first term tends to $\mathbf{\bar{f}}_1$, while the second term it propagates itself from the boundaries inside the field of view as $k$ grows.

In general it is not possible to estimate the difference $\mathbf{s}_{\textrm{out}}^{\mathcal{BC}} - \ddot{\mathbf{s}}_{\textrm{out}}$.
However it is reasonable to assume that the approximation of the real signal outside the boundaries $\mathbf{s}_{\textrm{out}}^{\mathcal{BC}}$ produces an error $\mathbf{s}_{\textrm{out}}^{\mathcal{BC}} - \ddot{\mathbf{s}}_{\textrm{out}}$ which is a constant function. In order to estimate such error function we can assume that its value equals $\chi=\max\left(\left|\mathbf{s}_{\textrm{in}}^{\mathcal{BC}}\right|\right) = \max\left(\left|\ddot{\mathbf{s}}_{\textrm{in}}\right|\right)$. This is clearly in most cases an overestimation.

Now the question is how the error $R (I-\ddot{W}_1^{\mathcal{P}})^{k} (\mathbf{s}_{\textrm{out}}^{\mathcal{BC}} - \ddot{\mathbf{s}}_{\textrm{out}})$ propagates inside the field of view during the iterations.
We can estimate an upper bound for the error inside the field of view at step $k$ of the iterations using the following formula
\begin{equation}\label{eq:err_UB}
   \textrm{err}_k = R (I-\ddot{W}_1^{\mathcal{P}})^{k} \mathbf{u} = R \left(I - {{k}\choose{1}}\ddot{W}_1^{\mathcal{P}} + {{k}\choose{2}}\left(\ddot{W}_1^{\mathcal{P}}\right)^2 +\ldots (-1)^{k} \left(\ddot{W}_1^{\mathcal{P}}\right)^k \right) \mathbf{u},
\end{equation}
where $\mathbf{u}\in\R^{n+2p}$ is a vector equal to $\chi$ outside the field of view and $0$ inside.

A potential upper bound of the error for each point $x_i$ inside the field of view can be computed as
\begin{equation}\label{eq:err_UB2}
   \textrm{ub}_k (x_i) = \max_{j\leq k}\left(\left|\textrm{err}_j(x_i)\right|\right),
\end{equation}
where $k$ is the number of iterations.
This is just an estimate because  in general the actual error introduced by the extension of the signal outside the boundaries is unknown.
If we deal with artificial examples,
and we consider DIF algorithm employing some BCs,
we can compute the actual error using the formula
\begin{equation}\label{eq:err}
   \textrm{err}^{\mathcal{BC}}_k (x_i) = \left| \mathbf{f}_1 (x_i) - \mathbf{\bar{f}}_1 (x_i)  \right|,
\end{equation}
where $k$ is the number of iterations required to produce the approximated IMF $\mathbf{f}_1$,
while $\mathbf{\bar{f}}_1$ is the first exact IMF.

\FloatBarrier

\section{Numerical results}
\label{sec:Examples}

In order to show how DIF method works when we employ different BCs,
we present four simple examples in which the signal is artificially created through one IMF plus a trend.
In particular, we analyze how the errors introduced outside the boundaries propagate inside the field of view, evaluating the performance of formula \eqref{eq:err_UB} for the a priori estimation of the error when there is no knowledge on the behavior of the signal outside the boundaries.

\subsection{Example 1}

We consider a first signal, shown on the left of Figure \ref{fig:Test_ex_1},
which can safely extended periodically and anti-reflectively outside the boundaries.
On the right of Figure \ref{fig:Test_ex_1}
we can look at decomposition obtained after 3 steps of DIF method with Reflective BCs.
If we apply the DIF algorithm using Periodic, Reflective and Anti-Reflective BCs,
we can extract a first IMF and measure the difference, after a fixed number of steps,
between the exact IMF and the computed IMF measured in absolute.
In Figure \ref{fig:Test_ex_1b} we plot the errors introduced by each extension,
formula \eqref{eq:err},
and we compare them with the a priori estimate of the error computed using formula \eqref{eq:err_UB2}.
Due to a symmetry in the signal, it is sufficient to plot the errors for the left half of the field of view.
As expected,
we produce a minimal error in the decomposition if we extended the signal periodically or anti-reflectively. Furthermore,
the a priori upper bound estimate of the error proves to be correct.

\begin{figure}[h!]
    \centering
    \includegraphics[width=0.48\textwidth]{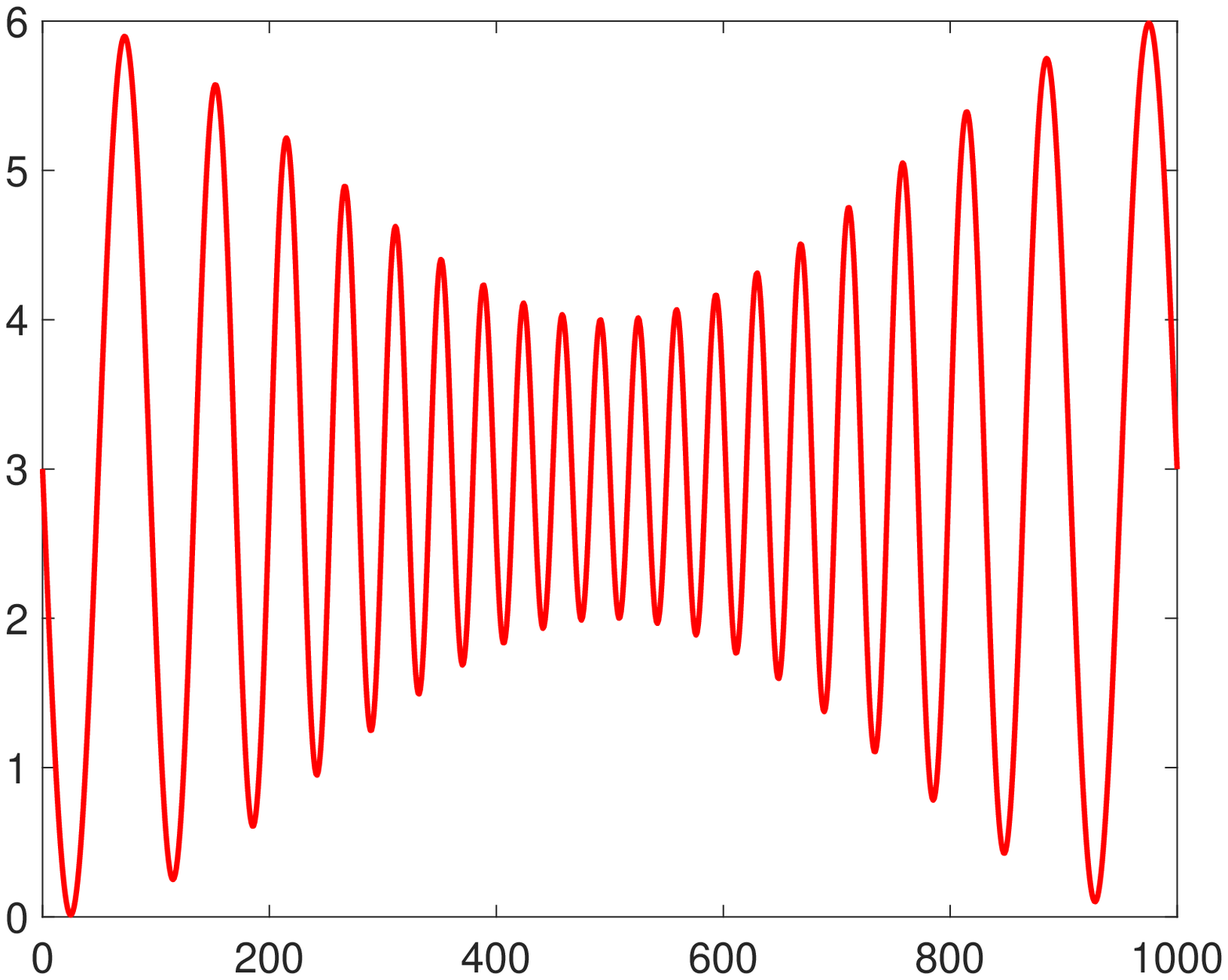}
    \includegraphics[width=0.48\textwidth]{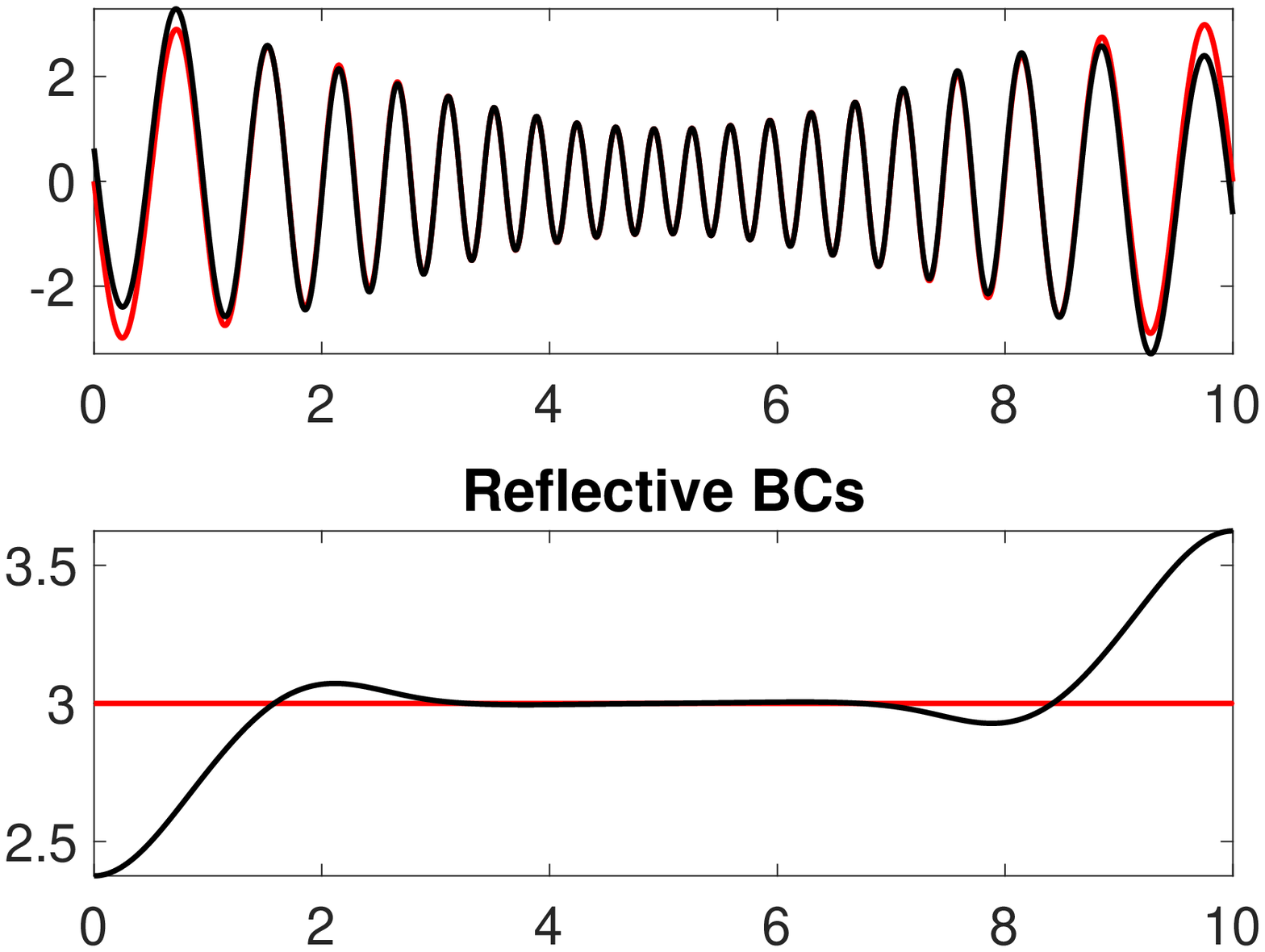}
    \caption{(Example 1) Left panel: signal. Right panel: decomposition computed by DIF method with Reflective BCs (black line) compared with exact one (red line).}\label{fig:Test_ex_1}
\end{figure}

\begin{figure}[h!]
    \centering
    \includegraphics[width=0.48\textwidth]{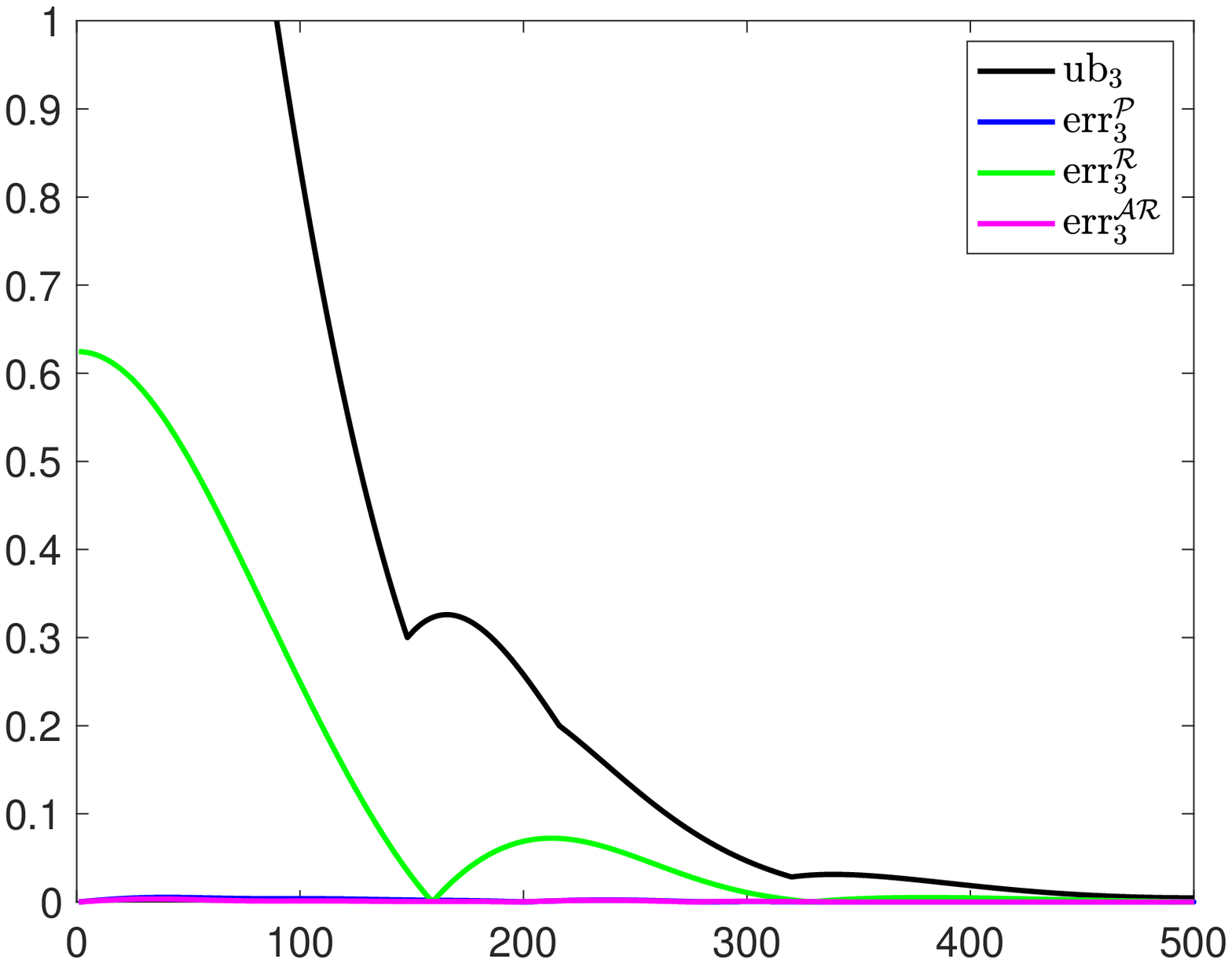}
    \includegraphics[width=0.48\textwidth]{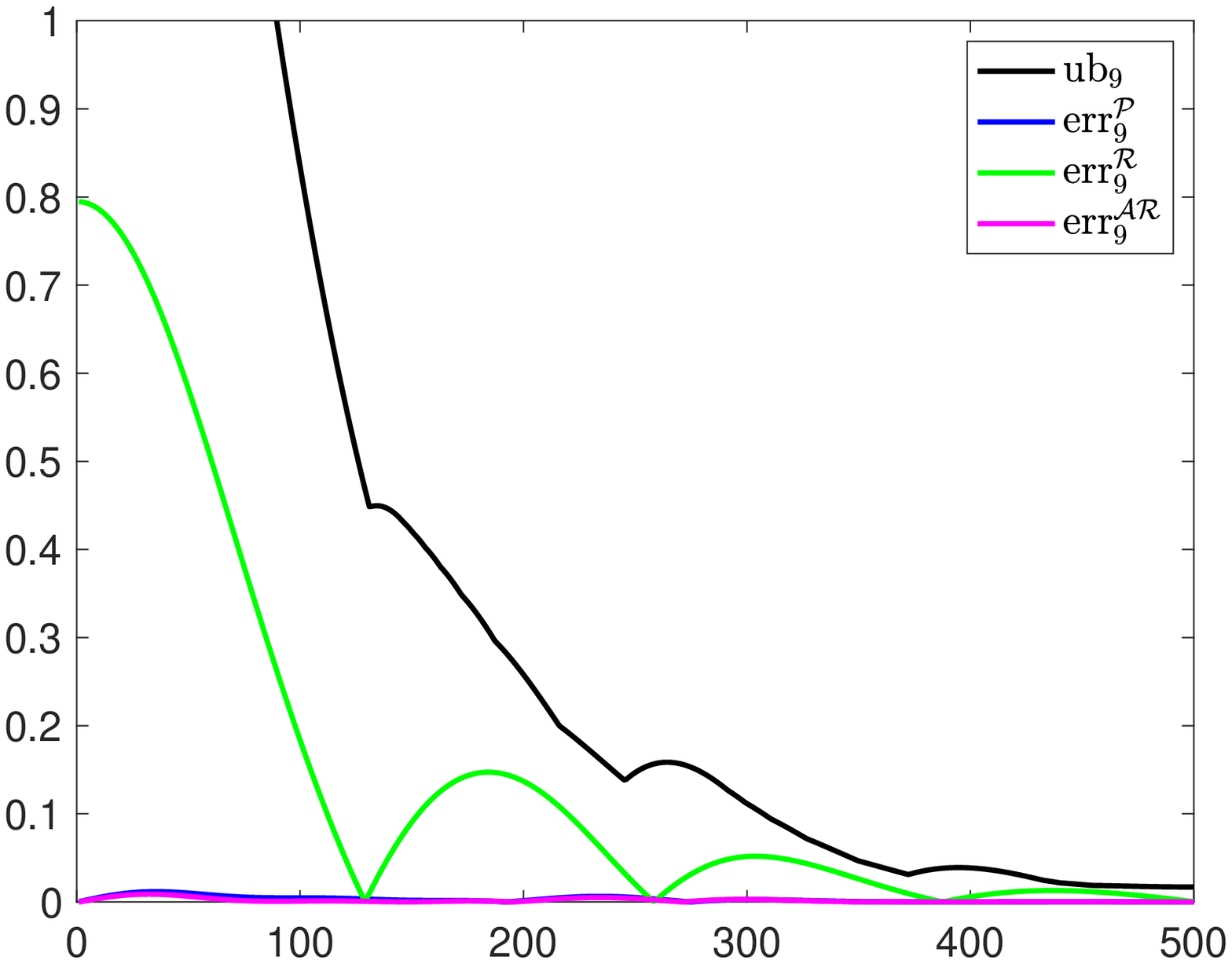}
    \caption{(Example 1) Upper bound and errors measured in absolute value for Periodic, Reflective and Anti-Reflective BCs relative to 3 iterations (on the left) and 9 iterations (on the right) of DIF method.}\label{fig:Test_ex_1b}
\end{figure}

\FloatBarrier

\subsection{Example 2}

In this second example we consider the signal plotted in the left panel of Figure \ref{fig:Test_ex_2},
while in the right panel
we can look at decomposition obtained after 3 steps of DIF method with Anti-Reflective BCs.
After extending the signal using Periodic, Reflective and Anti-Reflective BCs,
we extract a first IMF using the DIF method.
In Figure \ref{fig:Test_ex_2b} we compare the absolute values of the errors introduced by the different BCs, formula \eqref{eq:err}, and we compute the upper bound by means of \eqref{eq:err_UB2}.
Again, due to a symmetry in the curves, we report only the first half of the error plot.
In this case the reflective extension works best in minimizing the error in the decomposition.
Furthermore, the upper bound on the error allows to correctly estimate the error everywhere inside the field of view.

\begin{figure}[h!]
    \centering
    \includegraphics[width=0.48\textwidth]{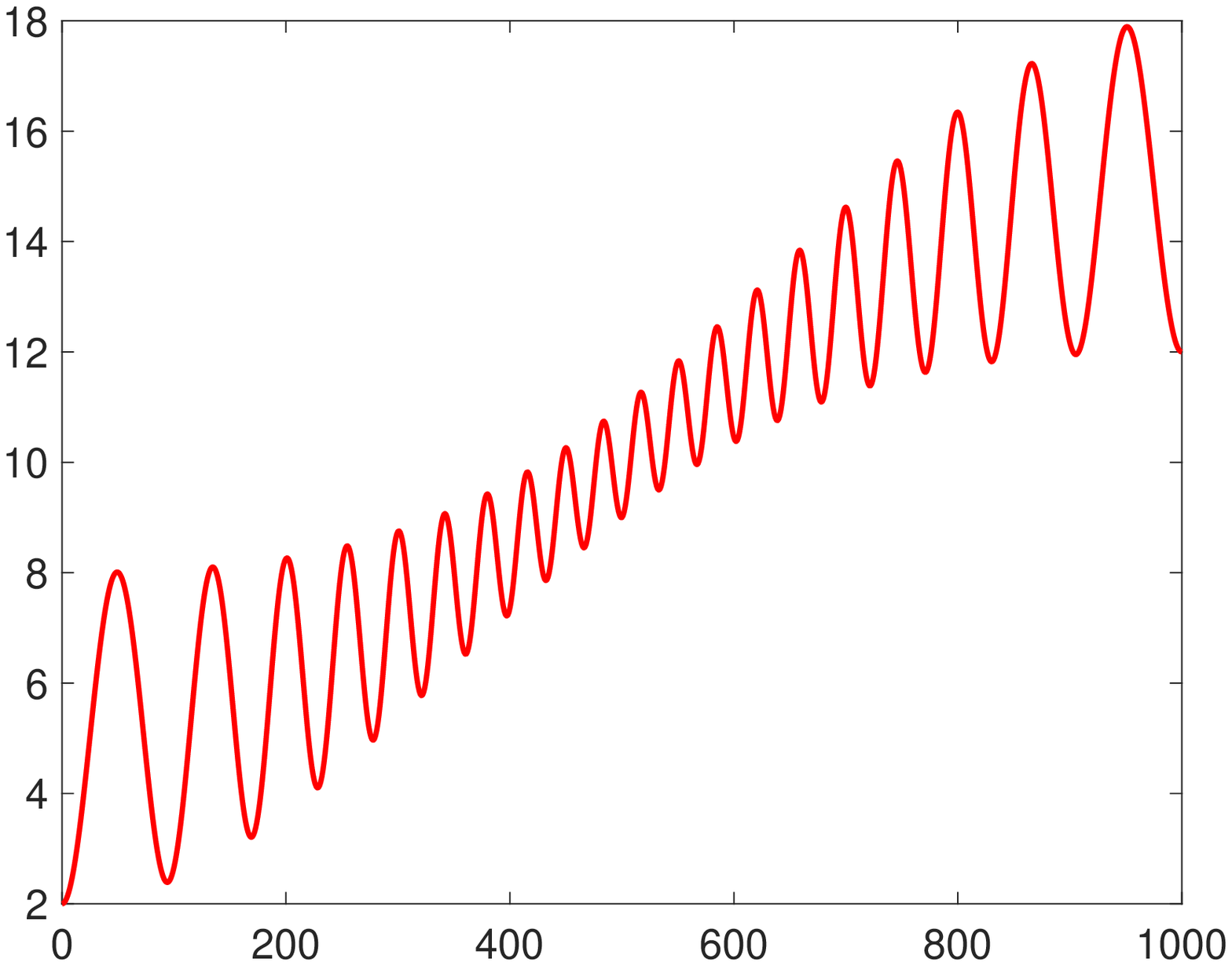}
    \includegraphics[width=0.48\textwidth]{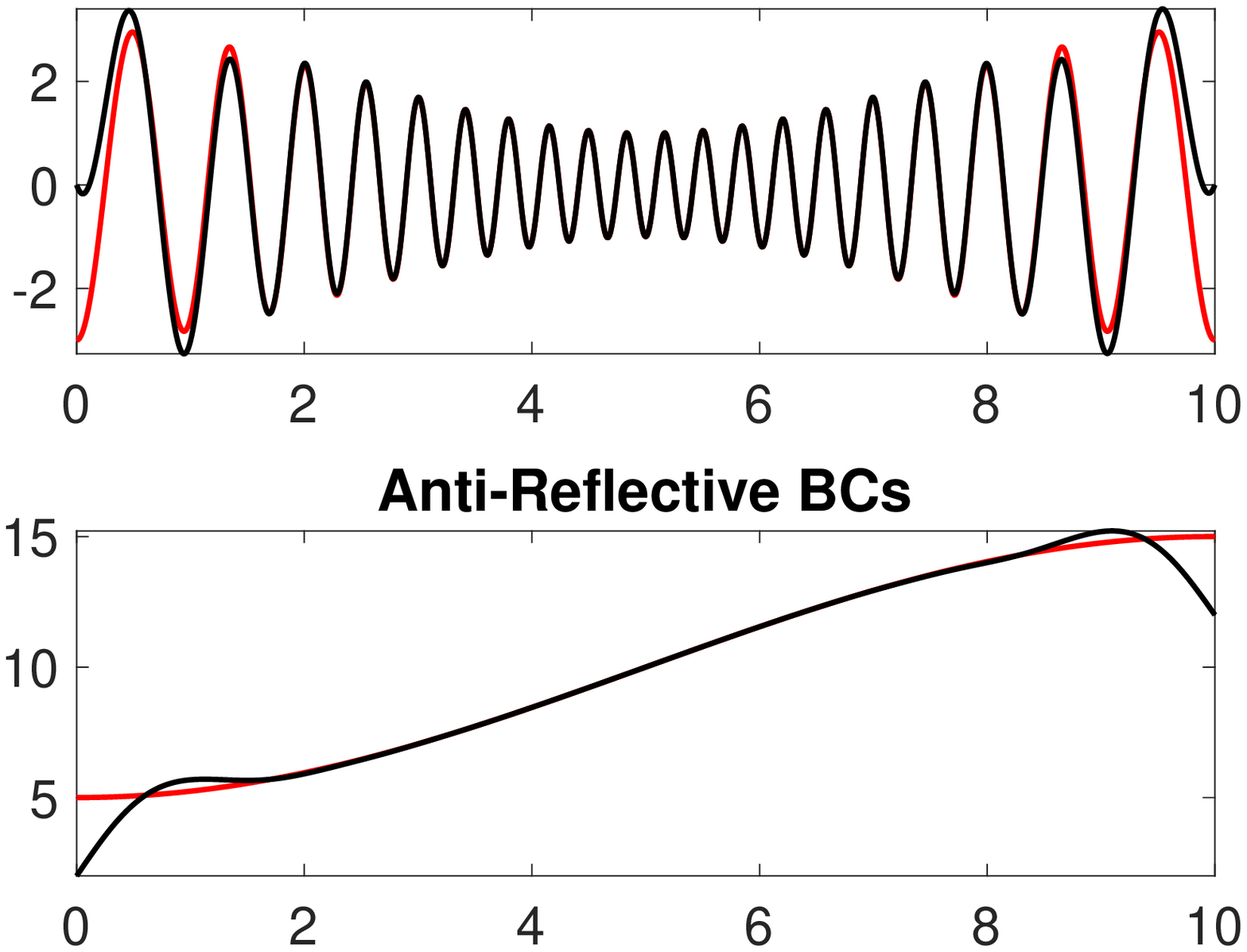}
    \caption{(Example 2) Left panel: signal. Right panel: decomposition computed by DIF method with Anti-Reflective BCs (black line) compared with exact one (red line).}\label{fig:Test_ex_2}
\end{figure}

\begin{figure}[h!]
    \centering
    \includegraphics[width=0.48\textwidth]{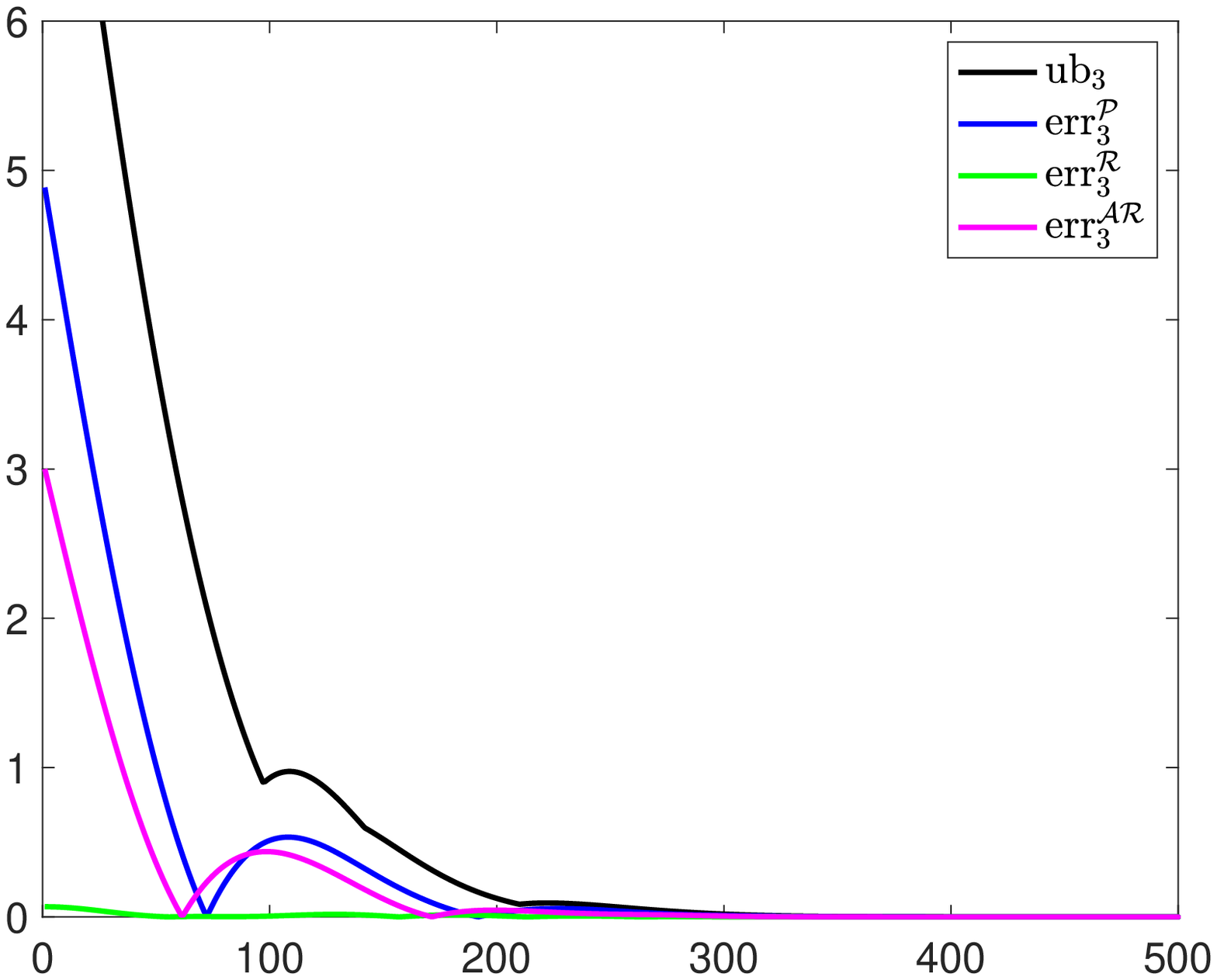}
    \includegraphics[width=0.48\textwidth]{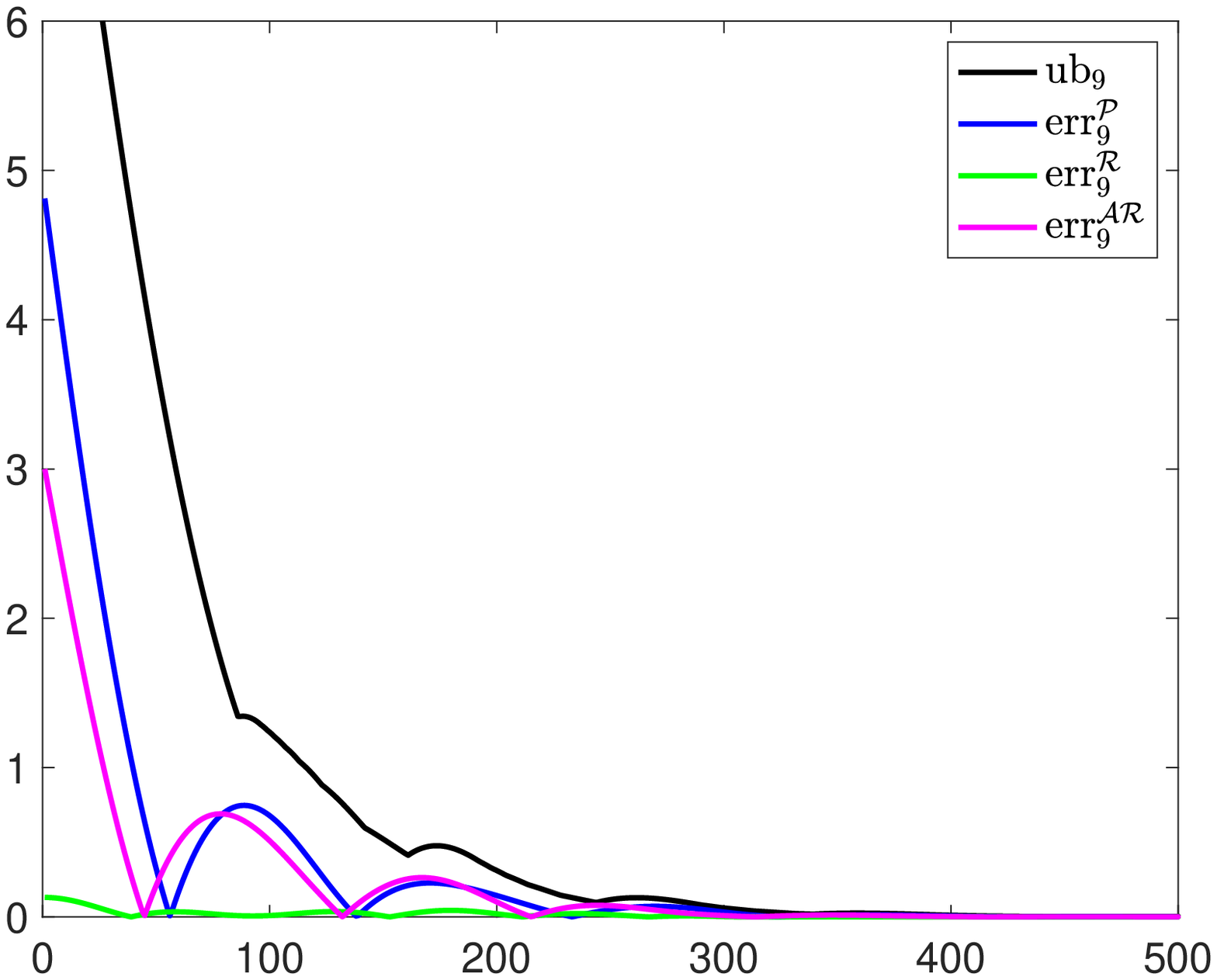}
    \caption{(Example 2) Upper bound and errors measured in absolute value for Periodic, Reflective and Anti-Reflective BCs relative to 3 iterations (on the left) and 9 iterations (on the right) of DIF method.}\label{fig:Test_ex_2b}
\end{figure}

\FloatBarrier

\subsection{Example 3}

In this case the signal, shown in the left panel of Figure \ref{fig:Test_ex_3}, is best extended outside the boundaries using Anti-Reflective BCs.
This fact is confirmed by the error curves \eqref{eq:err} plotted in Figure \ref{fig:Test_ex_3b}.
Due to symmetry properties of the signal under study, we report only the first half of the error curves.
Also in this example the upper bound allows to give a priori estimate of the behavior relative to the error propagation inside the field of view.

\begin{figure}[h!]
    \centering
    \includegraphics[width=0.48\textwidth]{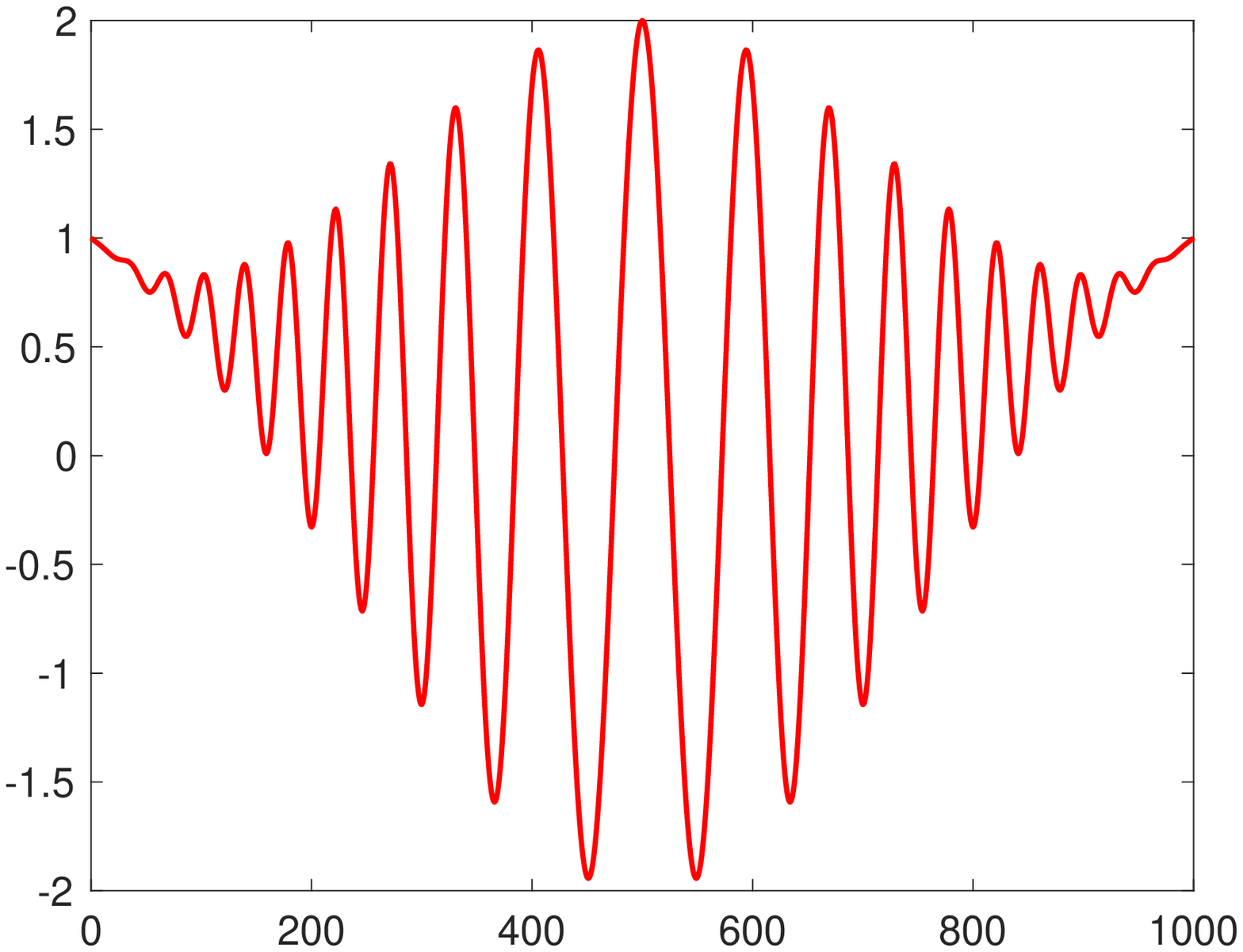}
    \includegraphics[width=0.48\textwidth]{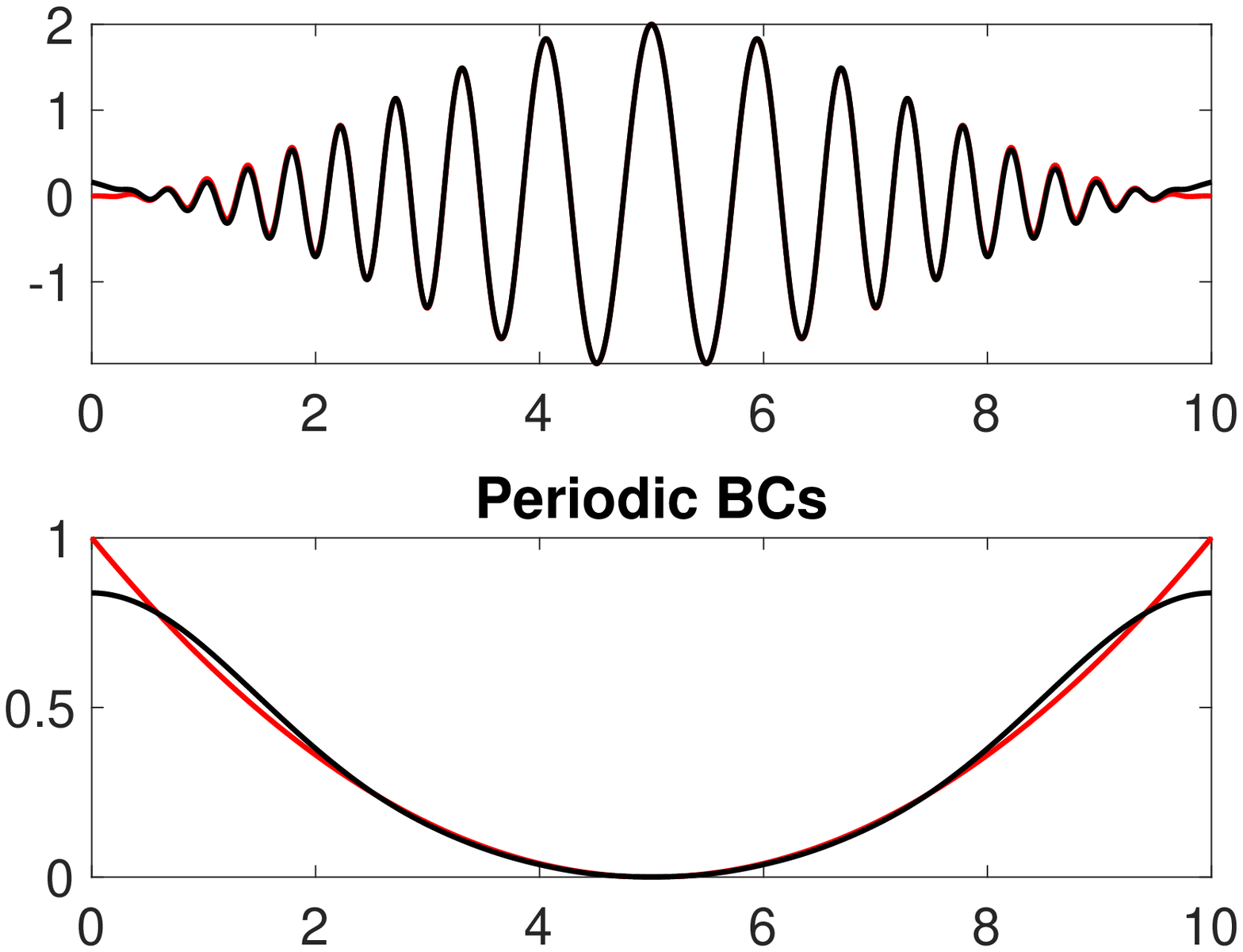}
    \caption{(Example 3)Left panel: signal. Right panel: decomposition computed by DIF method with Periodic BCs (black line) compared with exact one (red line).}\label{fig:Test_ex_3}
\end{figure}

\begin{figure}[h!]
    \centering
    \includegraphics[width=0.48\textwidth]{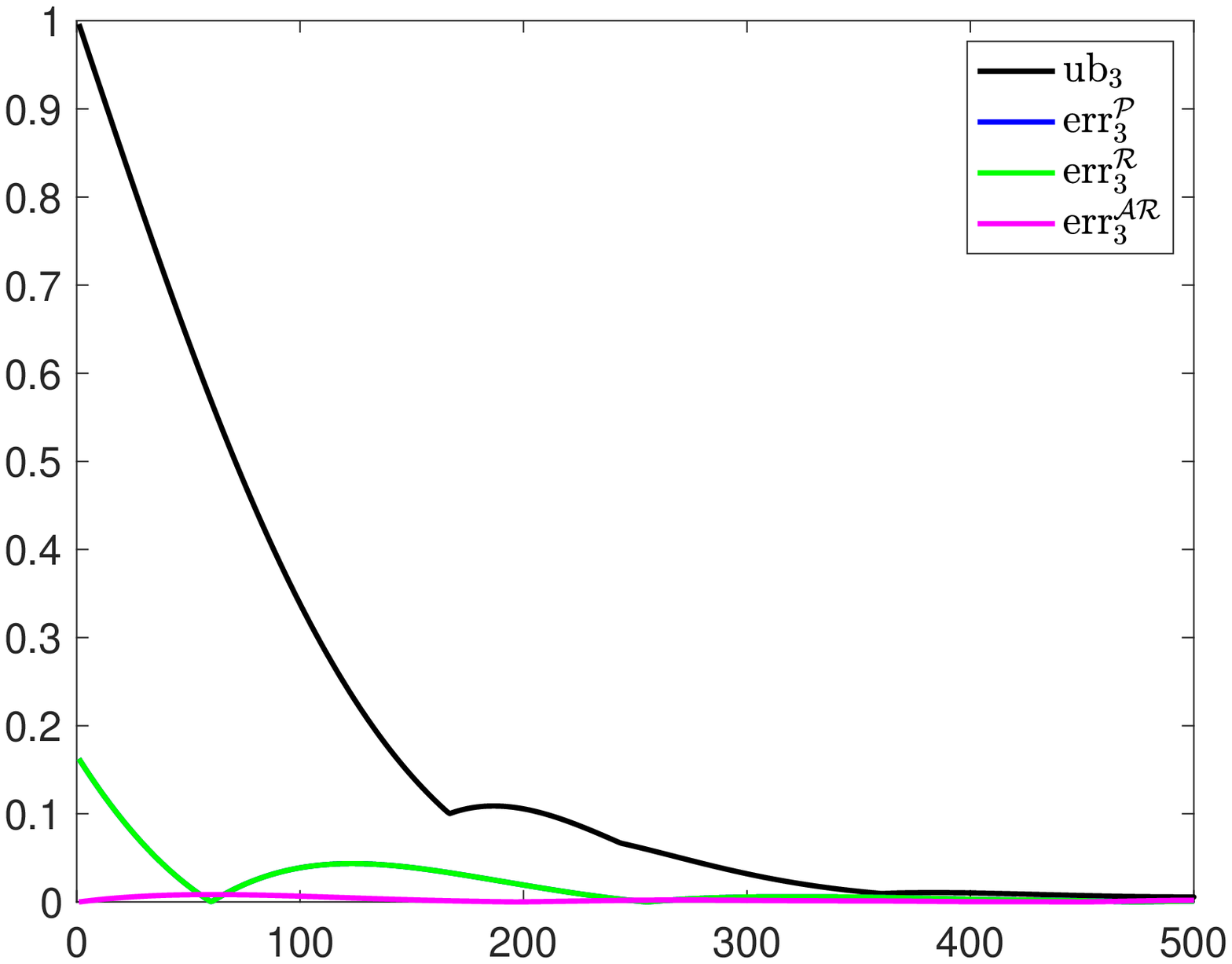}
    \includegraphics[width=0.48\textwidth]{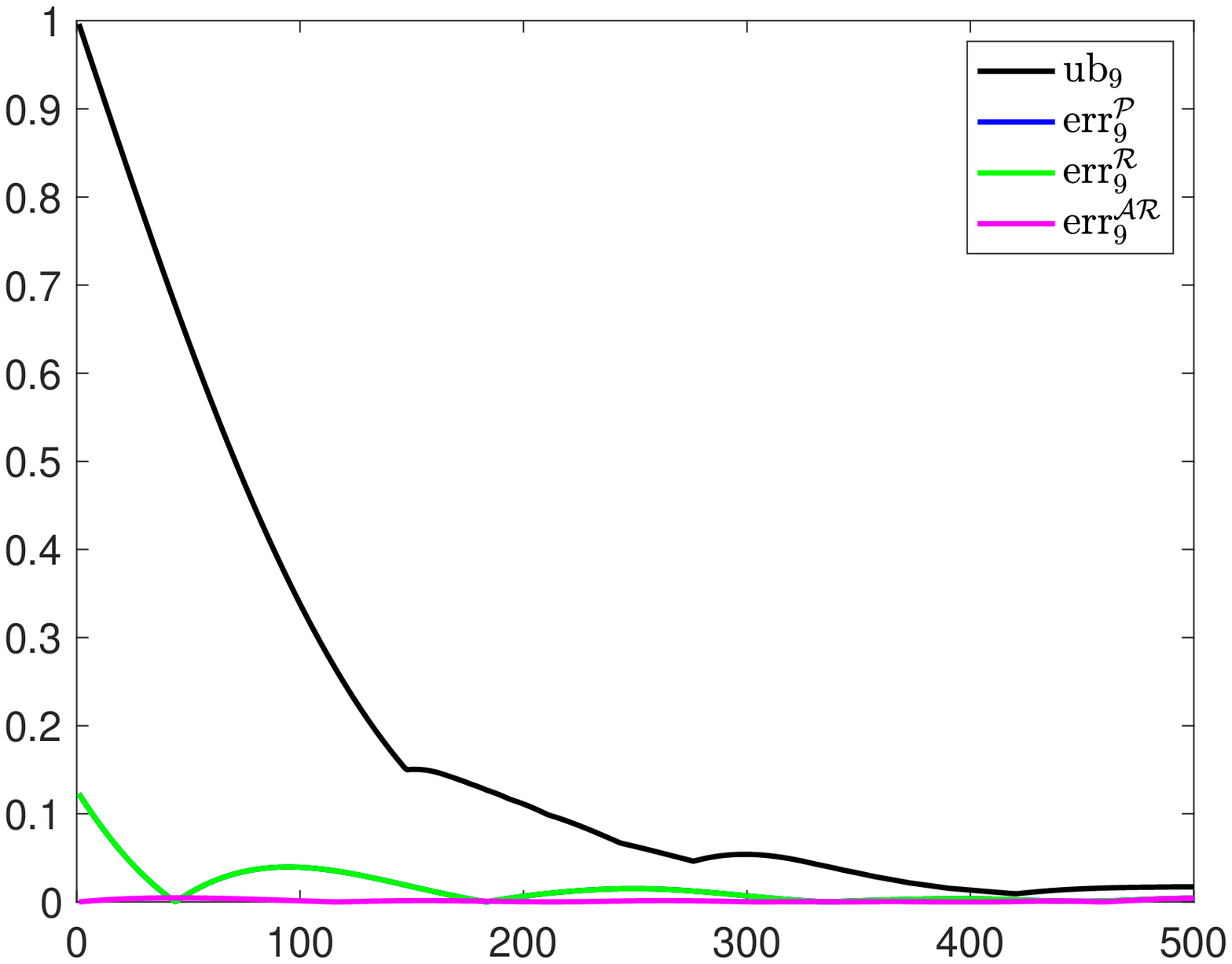}
    \caption{(Example 3)Upper bound and errors measured in absolute value for Periodic, Reflective and Anti-Reflective BCs relative to 3 iterations (on the left) and 9 iterations (on the right) of DIF method.}\label{fig:Test_ex_3b}
\end{figure}

\FloatBarrier

\subsection{Example 4}

In this example we want to study how the errors induced by the extension outside the boundaries depends on the phase of the signal at the boundary. We consider a simple signal, plotted in the left panel of Figure \ref{fig:Test_ex_4}, which is a superposition of a constant trend and a plain sine. The signal support is originally given by the interval $[-333,\ 0]$ which is then extended up to 300 step by step of $\triangle t = 0.01$. Each time we enlarge the support of a $\triangle t$ we redecompose the newly extended signal using DIF with different kind of BCs. For each BCs we compute the relative error
\begin{equation}\label{eq:rel_err}
    \textrm{err}^{\mathcal{BC}}_{\textrm{rel}} = \frac{\left\|\mathbf{f}_1 - \mathbf{\bar{f}}_1 \right\|_\infty}{\left\|\mathbf{\bar{f}}_1\right\|_\infty}
\end{equation}
and the relative error upper bound
\begin{equation}\label{eq:rel_err_UB}
    \textrm{ub}_{\textrm{rel}} = \frac{\left\|\textrm{ub}_k\right\|_\infty}{\left\|\mathbf{\bar{f}}_1 \right\|_\infty}
\end{equation}
where $\mathbf{\bar{f}}_1$ represents the first exact IMF and $\textrm{ub}_k$ is the error upper bound computed using \eqref{eq:err_UB2} for a fixed number of iterations $k$.
We plot the relative errors in the right panel of Figure \ref{fig:Test_ex_4}.
\begin{figure}[h!]
    \centering
    \includegraphics[width=0.48\textwidth]{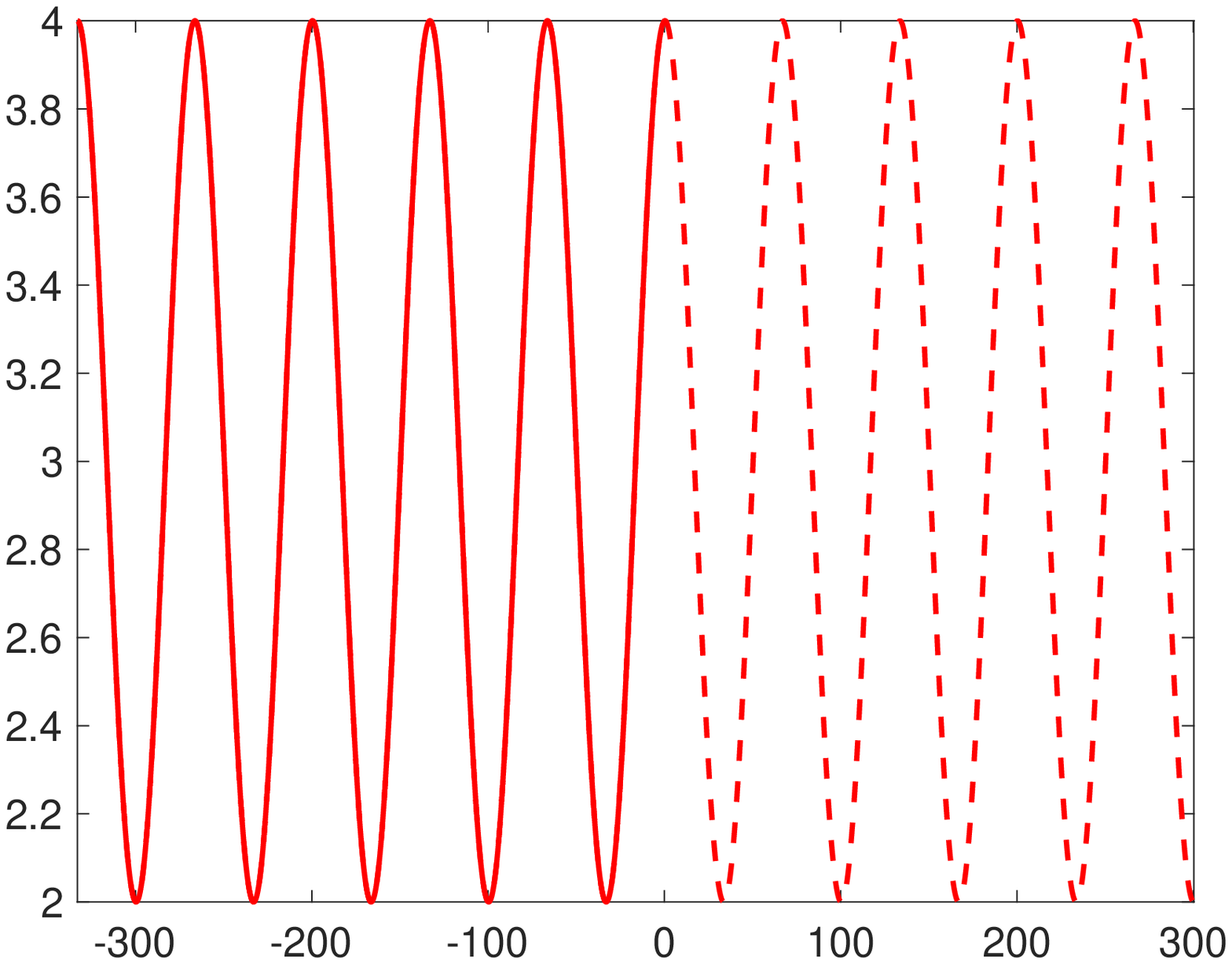}
    \includegraphics[width=0.48\textwidth]{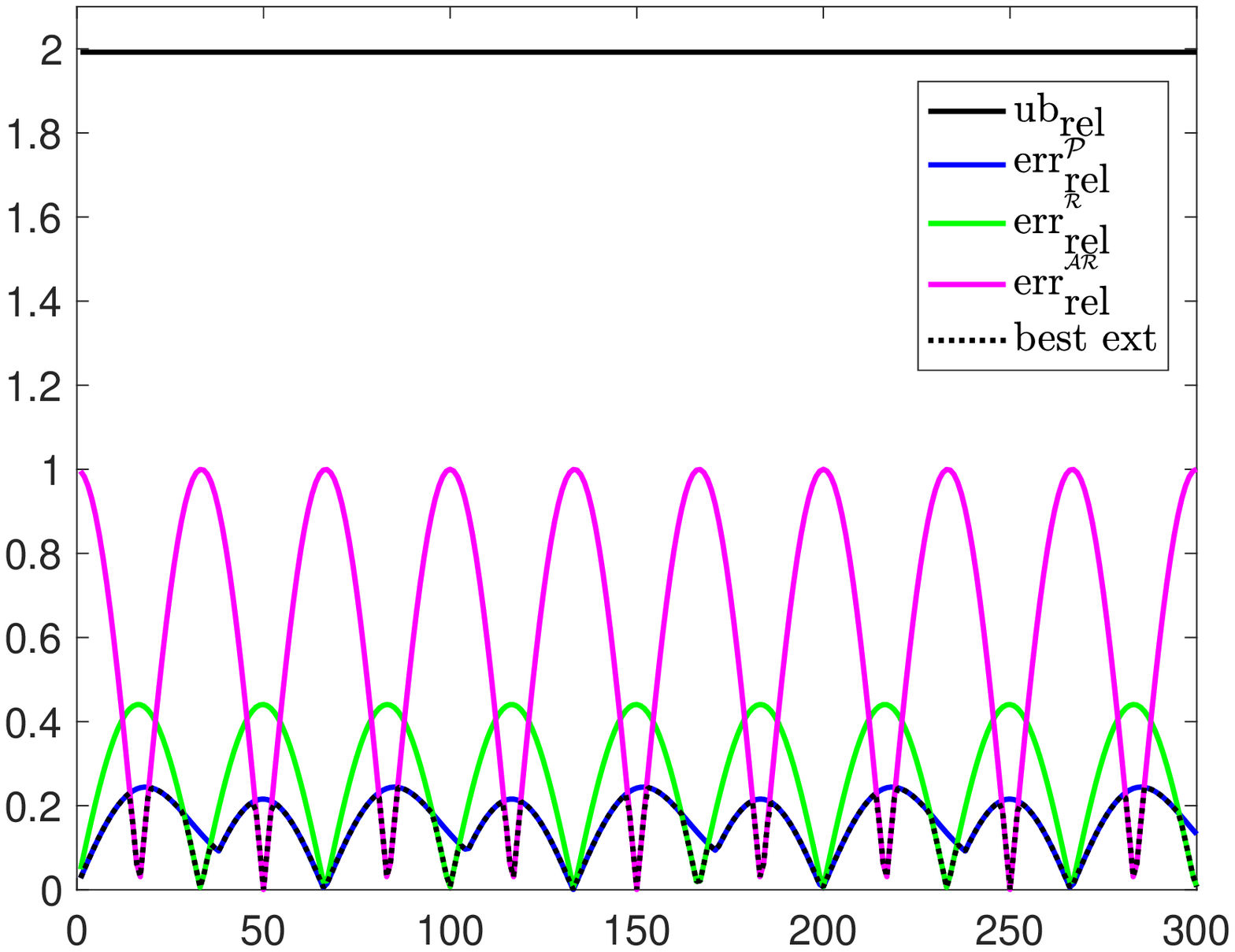}
    \caption{(Example 4) Left panel: signal. Right panel: upper bound, relative errors for different BCs and best extension.}\label{fig:Test_ex_4}
\end{figure}
As expected such curves have the same periodicity as the given signal. Furthermore the relative error upper bound a priori estimate proves to be valid also in this example. Finally this study allows to identify which kind of extension is performing best from a relative error point of view for each value of the phase of the signal at the boundary, dashed curve the right panel of Figure \ref{fig:Test_ex_4}.

\FloatBarrier

\section{Conclusions}\label{sec:Conclusions}

We considered the decomposition of non-stationary signals,
which is a problem of great interest from the theoretical point of view
and has important applications in many different fields.
For instance, it occurs in the identification of hidden periodicities and trends in time series relative to natural phenomena
(like average troposphere temperature)
and economic dynamics
(like financial indices).
Since standard techniques like Fourier or Wavelet Transform are unable to properly capture non-stationary phenomena,
in the last years several ad hoc methods have been proposed in the literature.
Such techniques provide iterative procedures for decomposing a signal into a finite number of simple components.

In this paper we focused on investigating IF algorithm employing different BCs
(Periodic, Reflective and Anti-Reflective BCs),
which give rise to different matrix structures.
We analyzed spectral properties of these matrices and convergence properties of IF method.
We also presented an extended version of IF, in which any BCs can be employed;
this allows to estimate the error propagation from the boundary towards the internal part of the signal.
Numerical experiments show that a suitable choice of BCs is able to improve in a meaningful way the quality
of signal decomposition in IMFs computed by IF method.

We think that this paper can open the way for further interesting developments,
for instance the study of adaptive choice of BCs, basing on the signal at hand,
and the proposal of new accurate BCs,
as done in \cite{pietro2017} for image restoration problem.

\section*{Acknowledgments}
This work was supported by the Istituto Nazionale di Alta Matematica (INdAM) ``INdAM Fellowships in Mathematics and/or Applications cofunded by Marie Curie Actions'', FP7--PEOPLE--2012--COFUND, Grant agreement n. PCOFUND--GA--2012--600198.

\end{document}